\documentclass[10pt]{article}

\usepackage{mathptmx,amsmath,amscd,amssymb,amsthm,xspace,enumerate,array,multirow}
\usepackage[all,tips]{xy}
\usepackage[dvips]{graphicx}
\usepackage{graphics,epsfig,wrapfig,verbatim,syntonly}
\usepackage{hyperref,amssymb,color,url,fancyhdr}

\usepackage{mathtools}

\theoremstyle{plain}

\newtheorem{thm}{Theorem}[section]

\newtheorem{cor}[thm]{Corollary}

\newtheorem{prop}[thm]{Proposition}
\newtheorem{propn}[thm]{Proposition}

\newtheorem{lemma}[thm]{Lemma}

\theoremstyle{definition}
\newtheorem*{defn}{Definition}

\newtheorem*{rmk}{Remark}

\newtheorem{ex}{Example}

\DeclareMathOperator{\PSL}{PSL}
\DeclareMathOperator{\SL}{SL}

\DeclareMathOperator{\lcm}{lcm}
\DeclareMathOperator{\D}{D}

\newcommand{\eps}{\varepsilon}
\newcommand{\vp}{\varphi}

\newcommand{\De}{\Delta}

\newcommand{\Ga}{\Gamma}

\newcommand{\La}{\Lambda}

\newcommand{\ol}{\overline}

\newcommand{\Hy}{\mathbf{H}}

\newcommand{\Q}{\mathbf{Q}}
\newcommand{\R}{\mathbf{R}}
\newcommand{\Z}{\mathbf{Z}}

\newcommand{\F}{\mathbf{F}}

\newcommand{\ds}{\displaystyle}
\newcommand{\mr}{\mathrm}
\newcommand{\mb}{\mathbf}
\newcommand{\mc}{\mathcal}
\newcommand{\on}{\operatorname}

\newcommand{\med}{\mr{mid}}
\newcommand{\Ab}{\mr{Ab}}
\newcommand{\m}{\mb{m}}
\newcommand{\n}{\mb{n}}
\newcommand{\x}{\mb{x}}

\newcommand{\cvec}{\mb{c}}
\newcommand{\dvec}{\mb{d}}
\newcommand{\bet}{b_1}
\newcommand{\fq}{\mr{FQ}}

\newcommand{\Iplus}{\ensuremath{\on{Isom}^+}}

\newcommand{\col}{\ensuremath{\colon}}

\newcommand{\thisthmname}{}
\newtheorem{genericthm}[thm]{\thisthmname}

\newtheorem*{named}{\theoremname}
\newcommand{\theoremname}{}
\newenvironment{namedthm*}[1]{
    \renewcommand{\theoremname}{#1}
    \begin{named}}
    {\end{named}}

\newcounter{caseNum}
\newenvironment{caseof}{\setcounter{caseNum}{1}}{\vskip.5\baselineskip}
\newcommand{\case}[2]{\vskip.5\baselineskip\par\noindent {\bfseries Case \Roman{caseNum}:} #1\\#2\addtocounter{caseNum}{1}}

\fancypagestyle{plain}

\begin{document}
\bibliographystyle{plain}


\title{\textbf{Finite quotients of Fuchsian groups}}
\author{Frankie Chan and Lindsey Styron}
\maketitle

\begin{abstract}
    \noindent This work provides an effective algorithm for distinguishing finite quotients between two non-isomorphic finitely generated Fuchsian groups $\Ga$ and $\La$. It will suffice to take a finite quotient which is abelian, dihedral, a subgroup of $\PSL(2,\F_q)$, or an abelian extension of one of these 3. We will develop an approach for creating group extensions upon a shared finite quotient of $\Ga$ and $\La$ which between them have differing degrees of smoothness. Regarding the order of a finite quotient that distinguishes between $\Ga$ and $\La$, we establish an upperbound as a function of the genera, the number of punctures, and the cone orders arising in $\Ga$ and $\La$.
\end{abstract}

\section{Introduction}

Bridson--Conder--Reid \cite{BCR} proved that finitely generated Fuchsian groups can be distinguished from lattices of connected Lie groups via their finite quotients. Moreover, they provided an explicit description for distinguishing finite quotients between two non-isomorphic (cocompact) triangle groups. We are motivated in extending this result; we prove that finite generated Fuchsian groups can be distinguished from each other by their finite quotients with an effective algorithm.

\begin{namedthm*}{Theorem \ref{thm:effective_bounds}}
    Let $\Ga=(g_1;p_1;\m)$ and $\La=(g_2;p_2;\n)$ be non-isomorphic Fuchsian groups, with $k:=|\m|=|\n|\geq0$ and $1,\infty\not\in\m\cup\n$. Set $M:=\lcm(\m)$, $N:=\lcm(\n)$, $L:=\lcm(M,N)$, and $b:=\max\{\bet(\Ga),\bet(\La)\}$. Then there exists a finite group $Q$ having order $$|Q|\ll (L+1)^{15+L^{15(b+k)}}$$ such that $Q$ is a quotient for one of the groups, but not for the other group.
\end{namedthm*}

\subsection{Profinite invariants}

For some background on the theory of profinite groups, see \cite{profiniteBook}.

Given a group $\Gamma$, define the set of \textit{finite quotients} of $\Gamma$ by $\fq(\Gamma):=\left\{[Q]\colon Q \text{ is a finite quotient of }\Gamma\right\}$, where $[Q]$ denotes the isomorphism class of the finite group $Q$.

\begin{defn}
    Let $\mc{F}$ be the set of finite index normal subgroups of $\Gamma$. Given $M,N\in\mc{F}$ such that $M\leq N$, let $\pi_M^N\col\Gamma/M\twoheadrightarrow\Gamma/N$ be the natural projection map. The \textit{profinite completion} $\widehat{\Gamma}$ of $\Gamma$ is the subgroup given by $$\displaystyle{\widehat{\Gamma}:=\left\{\left(x_M\right)_{M\in\mc{F}}\,\middle|\,x_M\in\Gamma/M\text{ and for every }M\leq N,\ \pi_M^N(x_M)=x_N\right\}\leq\prod_{M\in\mc{F}}\Gamma/M}.$$
\end{defn}

The profinite completion $\widehat{\Gamma}$ is a categorical inverse limit of finite groups under the discrete topology. We say $\Gamma$ and $\Lambda$ are \textit{profinitely equivalent} if $\widehat{\Gamma}\cong\widehat{\Lambda}$. A deep result of Nikolov--Segal \cite{nikolov-segal} makes precise an intimate connection between $\widehat{\Gamma}$ and $\fq(\Gamma)$ for finitely generated groups.

\begin{thm}
    Let $\Gamma$ and $\Lambda$ be finitely generated groups, then $\widehat{\Gamma}\cong\widehat{\Lambda}$ if and only if $\fq(\Gamma)=\fq(\Lambda)$.
\end{thm}

If $\Ga_K$ is the intersection of all of the finite index normal subgroups of $\Gamma$, then $\fq(\Gamma)=\fq(\Gamma/K)$. To ensure we are not working with essentially redundant profinitely equivalent groups, we restrict $\Gamma$ to be residually finite.

\begin{defn}
    A group $\Gamma$ is \textit{residually finite}, if the intersection of all of the finite index normal subgroups of $\Gamma$ is the trivial subgroup.
\end{defn}

In particular, the group $\Gamma/\Ga_K$ is residually finite. Fuchsian groups, or more generally lattices of constant curvature symmetric 2-spaces are residually finite \cite{malcev}. We now define two different scopes of profinite rigidity. 

\begin{defn}
    Let $\mc{C}$ denote some subclass of the class of all residually finite groups.
    \begin{enumerate}
        \item A group $\Gamma\in\mc{C}$ is \textit{relatively profinitely rigid} in $\mc{C}$, if whenever $\widehat{\Gamma}\cong\widehat{\Lambda}$ for some $\Lambda\in\mc{C}$, then $\Gamma\cong\Lambda$.
    
        \item A residually finite group $\Gamma$ is \textit{(absolutely) profinitely rigid}, if $\Gamma$ is relatively profinitely rigid in the class of all residually finite groups.
    \end{enumerate}
\end{defn}

Thus, Theorem \ref{thm:effective_bounds} is establishing a relative profinite rigidity among the class of all finitely generated Fuchsian groups.

\subsection{Fuchsian groups}

For a more detailed introduction to Fuchsian groups, see \cite{katok}.

A \emph{Fuchsian group} $\Gamma$ is a discrete subgroup of $\PSL(2,\R)\cong\Iplus(\Hy^2)$, the orientation-preserving isometries of the hyperbolic plane $\Hy^2$. We will primarily work with the abstract presentation of a finite generated Fuchsian group $\Ga$, and unless otherwise stated, we will reserve the term \textit{Fuchsian group} to mean a non-elementary finitely generated Fuchsian group. The \textit{signature} and \textit{standard presentation} of a Fuchsian group $\Ga$ is associated to the 2-orbifold $\mc{O}_\Ga:=\Hy^2/\Ga$ and can be categorized in the following way. Given a possibly empty multiset $\m=(m_1,\dots,m_k)$ of \textit{cone orders} $1\leq m_i<\infty$, with \textit{genus} $g\geq0$, and number of \textit{punctures} $p\geq0$, we write \begin{align*}
    \Ga=(g;p;\m)=\langle\alpha_1,\beta_1,\dots,\alpha_g,\beta_g;y_1,\dots,y_p;x_1,\dots,x_k\mid[\alpha_1,\beta_1]\dots[\alpha_g,\beta_g]y_1\dots y_px_1\dots x_k=x_i^{m_i}=1\rangle,
\end{align*} provided that the orbifold \textit{Euler characteristic} $\chi(g;p;\m)$ satisfies $$\chi(g;p;\m):=2-2g-p-\sum_{i=1}^k\left(1-\frac{1}{m_i}\right)<1.$$ We call the elements $\alpha_r,\beta_s,y_t,x_u$ the \textit{standard generators} of the standard presentation of $(g;p;\m)$; in particular, $\alpha_r,\beta_s$ are \textit{hyperbolic}, $y_t$ are \textit{parabolic}, and $x_u$ are \textit{elliptic} generators. Instances of the conepoint $1\in\m$ may be removed from $\m$, and if we allow for $\m$ to contain $\infty$, then each instance of $\infty$ may removed form $\m$ and added to the number of punctures $p$. These alterations change neither the isomorphism type nor the Euler characteristic of $(g;p;\m)$.

When $(g;p;\m)$ has $p>0$ punctures, it is a virtually free group, and $(g;p;\m)\cong F_{2g+p-1}*\Z_{m_1}*\dots*\Z_{m_k}$ is a free product of the free group $F_r$ of rank $r=2g+p-1$ with cyclic groups $\Z_{m_i}$ of order $m_i$. Furthermore, since $(g;p;\m)\cong(g-1;p+2;\m)$ when $g,p>0$, we can always find a genus zero representative $(0;p';\m)$ isomorphic to $(g;p;\m)$. In this case, the Euler characteristic $\chi(g-1;p+2;\m)=\chi(g;p;\m)$ is also preserved, even though the associated orbifolds are not isometric. Outside of these isomorphisms mentioned, different hyperbolic signatures produce non-isomorphic Fuchsian groups. In the unpunctured case $p=0$, we have that $(g;0;\m)$ is virtually a surface group $\Sigma_{g'}=(g';0;-)$, which follows from Selberg \cite{selberg}. An important class of Fuchsian groups arises for $g=p=0$. The $(r,s,t)$\textit{-triangle group} is $\Delta(r,s,t):=(0;0;r,s,t)$. Given $\m$ of any length, we set $\De(\m):=(0;0;\m)$ and $\chi(\m):=\chi(0;0;\m)$, even if the group is not Fuchsian.

The definition of the Euler characteristic can apply to more than just Fuchsian groups. For convenience, we will allow for $\chi$ to be defined for signatures where $\chi$ is no longer an isomorphism invariant, e.g., for bad orbifolds. If $\m$ contains neither 1 nor $\infty$, we say that $\m$ or $\De(\m)$ is \textit{bad} if $|\m|=1$, or if $|\m|=2$ but $m_1\neq m_2$, otherwise, we say that $\m$ or $\De(\m)$ is \textit{good}. Good and bad orbifolds were introduced by Thurston \cite{thurston}.

Suppose $\chi(g;p;\m)<0$ and $1,\infty\not\in\m$. The following lists the unique isomorphism classes of Fuchsian groups. To ease the listing of fewer cases, even though they are non-hyperbolic, we include the groups $(0;2;-)\cong\Z$ and $(1;0;-)\cong\Z^2$.
\begin{enumerate}
    \item $(0;p;\m)$, where $p\geq2$.
    \item $(0;1;\m)$, where $|\m|\geq2$.
    \item $(g;0;\m)$, where $g\geq1$.
    \item $(0;0;\m)$, where $|\m|\geq3$.
\end{enumerate} We call Fuchsian groups of the forms 1 and 2 \textit{punctured} or \textit{non-compact}, and those of the forms 3 and 4 \textit{unpunctured} or \textit{compact}. We will also refer to the notion of punctured and unpunctured Fuchsian groups by their \textit{puncture type}.

\subsection{Finite index subgroups of Fuchsian groups}\label{FIsubgroups}

Every finite index subgroup $\La=(g';p';\n)$ of a Fuchsian group $\Ga=(g;p;\m)$ is also Fuchsian group, and the Riemann--Hurwitz formula establishes much of the structure arising in $\La$.

\begin{thm}[Riemann-Hurwitz formula]
    Let $\Ga$ be a Fuchsian group, and $\La$ a finite index subgroup of $\Ga$. Then $\La$ is also a Fuchsian group and $\chi(\La)=[\Ga:\La]\chi(\Ga)$. 
\end{thm}

The Riemann--Hurwitz formula helps determine the signature of the kernel $K=\ker\pi=(g';p';\n)$ of a surjective map $\pi\colon(g;p;\m)\twoheadrightarrow G$ onto a finite group $G$, where $1,\infty\not\in\m$. Take the standard parabolic and elliptic generators $y_1,\dots,y_p;x_1,\dots,x_k\in(g;p;\m)$ mapping to the elements $\pi(y_1),\dots,\pi(y_p);\pi(x_1),\dots,\pi(x_k)\in G$, and record their respective orders in $G$ with the multiset $\n:=(d_1,\dots,d_p;c_1,\dots,c_k)$.

By considering the elliptic cycles (see \cite{katok}), the kernel $K=(g';p';\n)$ admits the following multiset of (possibly infinite) cone orders $$\left(\infty^{(|G|/d_1)},\dots,\infty^{(|G|/d_p)};\left(\frac{m_1}{c_1}\right)^{(|G|/c_1)},\dots,\left(\frac{m_k}{c_k}\right)^{(|G|/c_k)}\right),$$ where $a^{(i)}$ indicates the conepoint $a$ repeated $i$ times. If follows that $$p':=\frac{|G|}{d_1}+\dots+\frac{|G|}{d_p} \text{  and  } \n=\left(\left(\frac{m_1}{c_1}\right)^{(|G|/c_1)},\dots,\left(\frac{m_k}{c_k}\right)^{(|G|/c_k)}\right).$$

To obtain the value of the genus $g'$ of $K$, apply the Riemann--Hurwitz formula \begin{align*}
    \chi(g';p';\n) &= |G|\chi(g;p;\m)\\
    2-2g'-p'-\sum_{i=1}^k\left(1-\frac{c_i}{m_i}\right)\frac{|G|}{c_i} &= |G|\left(2-2g-p-\sum_{i=1}^k\left(1-\frac{1}{m_i}\right)\right).
\end{align*} This simplifies to the following equation which we can solve for the value of $g'$: \begin{align}
    \chi(g';p';-)=|G|\chi(g;p;\cvec).\label{eqn:finding_kernel_bettiNum}
\end{align}

Observe that every finite index subgroup $(g';p';\n)$ of an unpunctured group $(g;0;\m)$ remains unpunctured, i.e., $p'=0$. Similarly, every finite index subgroup $(g';p';\n)$ of a punctured lattice $(g;p;\m)$ is also punctured, i.e., $p'>0$ whenever $p>0$.

\section{Abelian quotients}

We describe the \textit{abelianization} of a Fuchsian group $(g;p;\m)$ using the \textit{invariant factor form} for finitely generated abelian groups. To proceed, we generalize the notions of $\gcd$'s and $\lcm$'s using a family of operators $\med_i(\m)$. Given a multiset $\m:=(m_1,\dots,m_k)$ not containing $\infty$, we define $\med_i(\m)$, for $1\leq i\leq k$.

\begin{defn} Let $\ell>0$ be prime and $d$ be a nonzero integer.
    \begin{enumerate}
        \item We write $\ell^a\mid\mid d$ for some integer $a\geq0$, if $\ell^a\mid d$ and $\ell^{a+1}\nmid d$. In this case, the \textit{$\ell$-adic valuation} of $d$ is $a$, and we denote this by $v_{\ell}(d)=a$. We define $v_{\ell}(0)=\infty$.

        \item For a nonzero rational number $q=x/y$, where $x,y\in\Z$, define $v_{\ell}(q)=v_{\ell}(x/y)=v_{\ell}(x)-v_{\ell}(y)\in\Z$, which is well-defined.
    \end{enumerate}
\end{defn}

For each prime $\ell$ dividing $\lcm(\m)$, suppose $\alpha_{\ell,1}\leq \alpha_{\ell,2}\leq\dots\leq\alpha_{\ell,k}$ are the valuations $v_{\ell}(m_1),v_{\ell}(m_2),\dots,v_{\ell}(m_k)$ written in increasing order.

\begin{defn}
    For each $1\leq i\leq k$, we define $\ds \med_i(\m):=\prod_{\ell\text{ prime}}\ell^{\alpha_{\ell,i}}$, the product of each of the $i$-th lowest prime powers appearing among the prime factorizations of each of the integers $m_1,\dots,m_k$.
\end{defn} In particular, we have $\med_1(\m)=\gcd(\m)$, $\med_k(\m) =\lcm(\m)$, $\ds \prod_{i=1}^k m_i=\prod_{i=1}^{k} \med_i(\m)$, and $\med_i(\m)\mid\med_{j}(\m)$ for $i<j$. We can now express the abelianization of $(g;p;\m)$ in the following piece-wise way.

\begin{propn}\label{prop:abelianization}
    The abelianization of a Fuchsian group $(g;p;\m)$ has the following invariant factor form (possibly including trivial $\Z_1$ factors) $$(g;p;\m)^{\Ab}\cong
    \begin{cases}
        \Z^{2g+p-1}\times\Z_{\med_1(\m)}\times\dots\times\Z_{\med_k(\m)} & \text{if } p>0\\
        \Z^{2g}\times\Z_{\med_1(\m)}\times\dots\times\Z_{\med_{k-1}(\m)} & \text{if } p=0.
    \end{cases}$$
\end{propn}

We can therefore deduce the following for Fuchsian groups with the same or mixed puncture types.

\begin{cor}\label{cor:abelianization}
    Suppose $\m=(m_1,\dots,m_k)$ and $\n=(n_1,\dots,n_k)$ are multisets, neither of which contains $\infty$. \begin{enumerate}
        \item If $(0;p_1;\m)^{\Ab}\cong(0;p_2;\n)^{\Ab}$ and $p_1,p_2>0$, then $p_1=p_2$ and $\med_i(\m)=\med_{i}(\n)$, for $1\leq i\leq k$. Furthermore, $m_1\dots m_k=n_1\dots n_k$.

        \item If $(g_1;0;\m)^{\Ab}\cong(g_2;0;\n)^{\Ab}$, then $g_1=g_2$ and $\med_i(\m)=\med_{i}(\n)$, for $1\leq i\leq k-1$.

        \item If $(g;0;\m)^{\Ab}\cong(0;p;\n)^{\Ab}$ with $p>0$, then $p=2g+1$, $\gcd(\n)=1$, and $\med_{i}(\m)=\med_{i+1}(\n)$, for $1\leq i\leq k-1$.
    \end{enumerate}
\end{cor}

\subsection{First Betti numbers}

A Fuchsian group $(g;p;\m)$ is \textit{torsion-free} if and only if there are integers $g',p'\geq0$ such that $(g;p;\m)\cong(g';p';-)$. Equivalently, $\m$ contains only 1 or $\infty$ among its cone orders.

\begin{defn}
    The \textit{first Betti number} of a finitely generated group $\Ga$ is the integer $b_1(\Ga):=\mr{rank}_{\Q}(\Ga^{\Ab}\otimes_{\Z}\Q)$.
\end{defn}

Recall the Euler characteristic $\ds \chi(g;p;\m)=2-2g-p-\sum_{i=1}^k\left(1-\frac{1}{m_i}\right)$. If $\m$ does not contain $\infty$, it follows from Proposition \ref{prop:abelianization} that $$b_1(g;p;\m)=\begin{cases}
    2g+p-1 & \text{if } p>0\\
    2g & \text{if } p=0,\\
\end{cases}$$ ignoring any instances of the torsion arising from the multiset $\m$. We have the following result.

\begin{prop}\label{prop:torsionfree_chi_determines_b1}
    The Euler characteristic of a torsion-free Fuchsian group $(g;p;-)$ determines its first Betti number. In particular, $$b_1(g;p;-)=\begin{cases}
        1-\chi(g;p;-) & \text{if } p>0\\
        2-\chi(g;0;-) & \text{if } p=0.
    \end{cases}$$
\end{prop}

\begin{propn}\label{prop:betti_bounds}
    Assume $\infty\not\in\m$. If $p>0$, then $b_1(g;p;\m)\leq1-\chi(g;p;\m)$, with equality if and only if $(g;p;\m)$ is torsion-free. If $p=0$, then $b_1(g;0;\m)\leq2-\chi(g;0;\m)$, with equality if and only if $(g;0;\m)$ is torsion-free.
\end{propn}

\begin{proof}
    Notice when $p>0$, then $b_1(g;p;\m)=b_1(g;p;-)=1-\chi(g;p;-)\leq1-\chi(g;p;\m)$, and similarly, $b_1(g;0;\m)=b_1(g;0;-)=2-\chi(g;0;-)\leq2-\chi(g;0;\m)$ for $p=0$, providing piecewise upperbounds on the first Betti numbers. Now, suppose $b_1(g;p;\m)$ attains its upperbound, depending on the value of $p\geq0$. Then $\chi(g;p;-)=\chi(g;p;\m)$, since $b_1(g;p;\m)=b_1(g;p;-)$. This implies that $\ds \sum_{i=1}^{k}\left(1-\frac{1}{m_i}\right)=0$, which cannot occur if $\m$ contains any conepoint which is greater than 1. Therefore, $(g;0;\m)=(g;0;-)$ is torsion-free.
\end{proof}

\subsection{Euler characteristics}\label{eulerChar}

Let $\m=(m_1,\dots,m_k)$ and $\n=(n_1,\dots,n_k)$ be multisets, neither of which contains $\infty$. By possibly concatenating or removing 1's, we may take $\m$ and $\n$ to have the same length.

The $\ell$-adic valuation satisfies the \textit{ultrametric inequality} $v_{\ell}(r+s)\geq\min\{v_{\ell}(r),v_{\ell}(s)\}$. Moreover, if $v_{\ell}(r)\neq v_{\ell}(s)$, then $v_{\ell}(r+s)=\min\{v_{\ell}(r),v_{\ell}(s)\}$. It follows that if $v_{\ell}(x_1)\leq v_{\ell}(x_2)\leq\dots\leq v_{\ell}(x_k)$, then $$v_{\ell}(x_1+x_2+\dots+x_k)\geq v_{\ell}(x_1),$$ with strict inequality only if $v_{\ell}(x_1)=v_{\ell}(x_2)$.

\begin{propn}
    If $(g;0;\m)^{\Ab}\cong(g;0;\n)^{\Ab}$ and $\chi(g;0;\m)=\chi(g;0;\n)$, then $\lcm(\m)=\lcm(\n)$ and $m_1\dots m_k=n_1\dots n_k$.
\end{propn}

\begin{proof}
    Corollary \ref{cor:abelianization} establishes that $\med_i(\m)=\med_i(\n)$, for $1\leq i\leq k-1$. The statement $m_1\dots m_k=n_1\dots n_k$ will hold true, once we argue that $\lcm(\m)=\lcm(\n)$. By way of contradiction, assume there is some prime $\ell$ such that $v_{\ell}(\lcm(\m))\neq v_{\ell}(\lcm(\n))$; without loss of generality, assume that $v_{\ell}(\lcm(\m))<v_{\ell}(\lcm(\n))$

    Set $\alpha_i:=v_{\ell}(\med_i(\m))=v_{\ell}(\med_{i}(\n))$ for $1\leq i\leq k-1$, $\beta:=v_{\ell}(\lcm(\m))$, and $\gamma:=v_{\ell}(\lcm(\n))$. We have $0\leq\alpha_1\leq\alpha_2\leq\dots\leq\alpha_{k-1}\leq\beta\lneq\gamma$. The condition on the Euler characteristics $\chi(g;0;\m)=\chi(g;0;\n)$ provides the following equation for the sum of reciprocals of the conepoints $$\underbrace{\frac{1}{m_1}+\dots+\frac{1}{m_k}}_{S}=\underbrace{\frac{1}{n_1}+\dots+\frac{1}{n_k}}_{T},$$ and observe that $v_{\ell}(S)\geq -\beta$ and $v_{\ell}(T)\geq -\gamma$ from the ultrametric inequality. Since $v_{\ell}(S)=v_{\ell}(T)$, we must have a strict inequality $v_{\ell}(T)>-\gamma$. It follows from another application of the ultrametric property that $\gamma=\alpha_{k-1}$, but this yields $\gamma=\alpha_{k-1}\leq\beta$, a contradiction. Therefore $v_{\ell}(\lcm(\m))=v_{\ell}(\lcm(\n))$ for every prime $\ell$.
\end{proof}

\begin{cor}\label{cor:lessthan2}
    Suppose that $\cvec,\dvec$ contains neither $1$ nor $\infty$, and let $|\cvec|=|\dvec|\leq2$. Consider the following cases.
    \begin{enumerate}
        \item Let $(g;0;\cvec)^{\Ab}\cong(g;0;\dvec)^{\Ab}$. Then $\chi(g;0;\cvec)=\chi(g;0;\dvec)$ if and only if $\cvec=\dvec$. 

        \item Let $(0;p;\cvec)^{\Ab}\cong(0;p;\dvec)^{\Ab}$ and $p>0$. Then $\chi(0;p;\cvec)=\chi(0;p;\dvec)$ if and only if $\cvec=\dvec$.
    \end{enumerate} 
\end{cor}

\begin{proof}
    The case where $|\cvec|=|\dvec|\leq1$ is clear, so assume $\cvec=(r,s)$ and $\dvec=(t,u)$. The condition on the Euler characteristics implies that $\frac{r+s}{rs}=\frac{t+u}{tu}$, and so we have $r+s=t+u$, since $rs=tu$ from Proposition \ref{prop:abelianization}. It follows that $(r,s)=(t,u)$, since both pairs represent the roots of the quadratic $x^2-(r+s)x+rs=x^2-(t+u)x+tu$.
\end{proof}

\section{Smooth and non-smooth quotients}

Let $G$ be a finite group. We say that a homomorphism $\vp\colon(g;p;\m)\to G$ is a \textit{smooth map} or \textit{smooth representation} of $(g;p;\m)$, if $\vp$ preserves the order of torsion elements in $(g;p;\m)$. In this case, we say that $Q:=\on{im}\vp$ is a \textit{smooth quotient} of $(g;p;\m)$. It suffices to check that the orders of only the standard elliptic generators of $(g;p;\m)$ are preserved under $\vp$. Equivalently, $G$ is a smooth quotient of $(g;p;\m)$ if $\ker\vp$ is torsion-free. Later in this section, we will develop more precise notions for maps which are not smooth.

\subsection{Smooth dihedral and \texorpdfstring{$\PSL_2$}{PSL2} representations}\label{psl}

We start by providing smooth dihedral and smooth $\PSL_2$ representations.

\begin{thm}\label{thm:smooth_dihedral}
    Let $\m=(m_1,\dots,m_k)$, with $k\geq1$, containing neither 1 nor $\infty$. There exists a finite dihedral group $D$ that is a smooth quotient of $(1;0;\m)$. Consequently, $D$ is a smooth quotient of $(g;p;\m)$ for $g\geq1$.
\end{thm}

\begin{proof}
    Take the following group presentation: $(1;0;\m)\cong\langle a,b;x_1,\dots,x_k\mid[a,b]x_1\dots x_k=x_i^{m_i}=1\rangle$. The dihedral group $\D_{2n}=\langle r,s\mid r^n=s^2=1,\ srs=r^{-1}\rangle$ of order $2n$ has commutator subgroup $[\D_{2n},\D_{2n}]=\langle r^2\rangle$ which is cyclic. In particular, when $n$ is odd, $\langle r^2\rangle$ has order $n$, and when $n$ is even, $\langle r^2\rangle$ has order $n/2$.

    If $M=\lcm(\m)$ is odd, then the following map is smooth: $\vp_M\colon(1;0;\m)\to\D_{2M}$ defined by $x_i\mapsto r^{M/m_i}$, $a\mapsto r^{\alpha}$, and $b\mapsto s$, where $\alpha:=-\frac{M+1}{2}\left(\frac{M}{m_1}+\dots+\frac{M}{m_k}\right)$. On the other hand, if $M$ is even, then the following map $\vp_{2M}\colon(1;0;\m)\to\D_{2\cdot(2M)}$ is smooth: $x_i\mapsto r^{2M/m_i}$, $a\mapsto r^{\beta}$, and $b\mapsto s$, where $\beta:=-\left(\frac{M}{m_1}+\dots+\frac{M}{m_k}\right)$. In either case, the images of $\vp_M$ and $\vp_{2M}$ are dihedral subgroups.
\end{proof}

We consider smooth representations of Fuchsian groups to $\PSL(2,q):=\SL(2,q)/\{\pm I_2\}$. This setting is extremely useful for Fuchsian groups $(g;p;\m)$ where $g=p=0$. The following result is from Macbeath \cite{macbeath}, and is repurposed for our needs.

\begin{thm}\label{thm:macbeath}
    Let $\m=(m_1,\dots,m_k)$ be a multiset such that $1,\infty\not\in\m$ and $k\geq3$. Suppose $q=\ell^e>1$ is an odd prime power. There exists a smooth representation of $\De(\m)$ into $\PSL(2,q)$ if and only if each of the integers $m_1,\dots,m_k$ divides one of $\ell,\ \frac{q-1}{2}$, or $\frac{q+1}{2}$.
\end{thm}

Retaining the hypothesis and notation from Theorem \ref{thm:macbeath}, notice there is an odd prime power $q>1$ such that each of the $m_i$ divides $\frac{q-1}{2}$. Take any odd prime $\ell$ that is coprime to $L:=m_1m_2\dots m_k$. Then $\ell$ is a unit in the ring $\Z/2L\Z$, and so there exists $e>0$ such that $\ell^e\equiv1\pmod{2L}$. Setting $q:=\ell^e$, we have our desired result.

\begin{rmk}
    Theorems \ref{thm:smooth_dihedral} and \ref{thm:macbeath} provide for an effective version of Selberg's lemma \cite{selberg} for Fuchsian groups, which states that a finitely generated linear group over a field of characteristic zero is virtually torsion-free.
\end{rmk}

\subsection{Degrees of smoothness of quotients}\label{smoothDeg}

The following theorem will be crucial for producing finite quotients for many of our situations.

\begin{thm}\label{thm:homologyTrick}
    Let $G$ be a finite group. If there is a surjective map $\pi\colon(g_1;p_1;\m)\twoheadrightarrow G$ such that every surjective map $s\colon(g_2;p_2;\n)\twoheadrightarrow G$ satisfies $b_1(\ker s)<b_1(\ker\pi)$, then there is an abelian extension $Q$ of $G$ such that $Q$ is a finite quotient of $(g_1;p_1;\m)$ but not for $(g_2;p_2;\n)$.
\end{thm}

\begin{proof}
    Take $K:=\ker\pi$, $f:=b_1(K)$, and $a>1$ an integer coprime to both $\lcm(\m)$ and $\lcm(\n)$. Consider the subgroup $L:=K^{(a)}[K,K]$, which is the join of the commutator subgroup $[K,K]$ with the subgroup $K^{(a)}$ generated by the $a$-th powers of elements in $K$. Since $L$ is a characteristic subgroup of $K$ and $K$ is a normal subgroup of $(g_1;p_1;\m)$, we have that $L$ is a normal subgroup of $(g_1;p_1;\m)$. Define $Q:=(g_1;p_1;\m)/L$. The surjective map $\pi\colon(g_1;p_1;\m)\twoheadrightarrow G$ naturally induces a map $\tilde{\pi}\colon Q\twoheadrightarrow G$ such that $\ker\tilde{\pi}=K/L\cong(\Z_a)^f$, which follows from our choice of $a$. This shows that $Q$ is a \textit{finite} quotient of $(g_1;p_1;\m)$.

    Suppose by way of contradiction, there is a surjection $t\colon(g_2;p_2;\n)\twoheadrightarrow Q$, and set $L':=\ker t$. From the hypothesis, the composite map $s=\tilde{\pi}\circ t\colon(g_2;p_2;\n)\twoheadrightarrow G$ with $K':=\ker s$ satisfies $b_1(K')<f$. Factoring $s$ through $(g_2;p_2;\n)/L'\cong Q$ yields a surjective map $\tilde{s}\colon Q\twoheadrightarrow G$ having kernel $K'/L'$.

    It follows that $K'/L'\cong K/L\cong(\Z_a)^f$. This gives us a surjective map $K'\twoheadrightarrow(\Z_a)^f$, which factors through the abelianization $(K')^{\Ab}$. This would imply that, as a result of how $a$ was chosen, that a surjective map $\Z^{b_1(K')}\twoheadrightarrow(\Z_a)^f$ exists, which is impossible since $b_1(K')<f$. Thus, $Q$ cannot be a quotient for $(g_2;p_2;\n)$.
\end{proof}

The hypothesis in Theorem \ref{thm:homologyTrick} can be thought of in the following way. While $G$ is a quotient for \textit{both} of the groups $(g_1;p_1;\m)$ and $(g_2,p_2;\n)$, the maximality condition on the first Betti numbers is detecting that $G$ is a ``smoother'' quotient for $(g_1;p_1;\m)$ than $G$ is as a quotient for $(g_2;p_2;\n)$. We now formalize this notion of the \textit{degree of smoothness} of a quotient $G$ for a Fuchsian group.

\begin{defn}
    Let $\m=(m_1,\dots,m_k)$ be a multiset, not containing $\infty$. We say that $\cvec=(c_1,\dots,c_k)$ is a \emph{factor (or divisor)} of $\m$, denoted $\cvec\mid\m$, if after a possible rearrangement of $\cvec$, we have $c_i\mid m_i$ for each $i$. There is a partial order on the set of factors of $\m$, where $\cvec\leq\dvec\xLeftrightarrow{\text{def}}\cvec\mid\dvec$.
\end{defn}

\begin{defn}
    Let $(g;p;\m)$ be a Fuchsian group, where $\infty\not\in\m$, and let $x_1,\dots,x_k\in(g;p;\m)$ be the standard elliptic generators with orders $m_1,\dots,m_k$, respectively. For a finite group $G$ and a factor $\cvec\mid\m$, define the following.
    \begin{enumerate}
        \item A homomorphism $\vp\colon(g;p;\m)\to G$ is a \textit{$\cvec$-smooth map (or representation)}, if the elements of $\cvec$ correspond to the orders of $\vp(x_1),\dots,\vp(x_k)$ in $G$. In this context, we also say that $Q:=\on{im}\vp$ is \textit{$\cvec$-smooth}.
    
        \item We say $G$ is a \textit{$\cvec$-smooth quotient of $(g;p;\m)$}, if there exists a $\cvec$-smooth surjective map $\pi\colon(g;p;\m)\twoheadrightarrow G$.

        \item We say $G$ is a \textit{$\cvec$-maximally smooth quotient of $(g;p;\m)$}, if $G$ is a $\cvec$-smooth quotient of $(g;p;\m)$ and the value of $\chi(\cvec)\in\Q$ is the least such among the values of $\chi(\cvec')$ for which $G$ is $\cvec'$-smooth. In such a case, we say that any $\cvec$-smooth surjection $\pi\colon(g;p;\m)\twoheadrightarrow G$ is \textit{$\cvec$-maximally smooth}.
    \end{enumerate}
\end{defn}

In particular, a homomorphism $\vp\colon(g;p;\m)\to G$ is $\m$-maximally smooth if and only if $\vp$ is smooth in the original sense.

\begin{cor}
    Let $\Ga,\La$ be non-isomorphic Fuchsian groups of the same puncture type such that $\Ga^{\Ab}\cong\La^{\Ab}$. Suppose $G$ is a finite group and there are surjective homomorphisms $\pi_1\colon\Ga\twoheadrightarrow G$ and $\pi_2\colon\La\twoheadrightarrow G$ which are $\cvec$-maximally smooth and $\dvec$-maximally smooth, respectively. If $\chi(\cvec)<\chi(\dvec)$, then there exists a finite abelian extension $Q$ of $G$ such that $Q$ is a finite quotient of $\Ga$ but not for $\La$.
\end{cor}

\begin{proof}
    This follows from Theorem \ref{thm:homologyTrick}, since equation (\ref{eqn:finding_kernel_bettiNum}), Proposition \ref{prop:torsionfree_chi_determines_b1}, and the assumption $\chi(\cvec)<\chi(\dvec)$ imply that $\bet(\ker\pi_1)>\bet(\ker\pi_2)$.
\end{proof}

To contrast the previous Corollary, the following proposition deals with Fuchsian groups with mixed puncture types.

\begin{prop}\label{prop:different_puncture_types}
    Let $\m,\n$ be (possibly empty) multisets, neither containing 1 nor $\infty$. Suppose $\Ga=(g;0;\m)$ and $\La=(0;p;\n)$ with $g,p>0$. Then there exists a finite quotient $Q$ for one of the Fuchsian groups but not for the other.
\end{prop}

\begin{proof}
    We compare Euler characteristics for the following two cases.
    \begin{caseof}
        \case{$\chi(g;0;\m)\leq\chi(0;p;\n)$}{Suppose that $\chi(g;0;\m)\leq\chi(0;p;\n)$. Then for any integer $l>0$, we have $\chi(g;0;\m)<\chi(0;p;\n)+1/l$ so that $l\chi(g;0;\m)<l\chi(0;p;\n)+1$ and $2-l\chi(g;0;\m)>1-l\chi(0;p;\n)$.
        
        Let $l$ be the order of any \textit{smooth} finite quotient $G$ of $(g;0;\m)$, and fix $K_1$ to be the associated kernel. Observe that $\bet(K_1)=2-l\chi(g;0;\m)$. If $G$ also happens to be a quotient of $(0;p;\n)$, let $K_2$ be any associated kernel and observe that $\bet(K_2)\leq 1-l\chi(0;p;\n)$. Therefore $\bet(K_1)>\bet(K_2)$, so by Theorem \ref{thm:homologyTrick} we can appropriately extend $G$ to a finite quotient $Q$ of $(g;0;\m)$ but not for $(0;p;\n)$.}

        \case{$\chi(g;0;\m)>\chi(0;p;\n)$}{Now suppose that $\chi(g;0;\m)>\chi(0;p;\n)$, and set $L:=\lcm(\m,\n)$. Since Euler characteristics are rational numbers, we have $\chi(g;0;\m)-\chi(0;p;\n)\geq1/L$. Hence, for every integer $l>L$, we have $2-l\chi(g;0;\m)<1-l\chi(0;p;\n)$.
    
        Let $G$ be a \textit{smooth} quotient of $(0;p;\n)$ having order $l>L$. Such a quotient $G$ exists because, for example, we can smoothly map a standard parabolic generator of $(0;p;\n)$ to an element of order, say $L+1$, in some $\PSL(2,q)$. Similar to the argument from the previous case, for this case, the associated kernels arising from $G$ give the inequality $\bet(K_1)<\bet(K_2)$. Then by Theorem \ref{thm:homologyTrick}, we can appropriately extend $G$ to a finite quotient $Q$ of $(0;p;\n)$ but not for $(g;0;\m)$.}
    \end{caseof}
\end{proof}

The following theorem is the final piece needed for the effective algorithm.

\begin{thm}\label{thm:distinguishing_factors}
    Suppose $\Ga=(g_1;p_1;\m)$ and $\La=(g_2;p_2;\n)$ are non-isomorphic Fuchsian groups. Then there exist factors $\cvec\mid\m$ and $\dvec\mid\n$ and a finite group $Q$ such that either $Q$ is an outright quotient for exactly one of $\Ga$ and $\La$, or $Q$ is a $\cvec$-maximally smooth quotient of $\Ga$ and a $\dvec$-maximally smooth quotient of $\La$ satisfying
    \begin{enumerate}
        \item $\chi(\cvec)\neq\chi(\dvec)$, and
        \item if $\Ga=\De(\m)$ and $\La=\De(\n)$, then at least one of $\cvec,\dvec$ is good.
    \end{enumerate}
\end{thm}

The proof of Theorem \ref{thm:distinguishing_factors} is devoted for the next section, whose aim is to produce the necessary factors $\cvec,\dvec$.

\section{Scrapes and smoothness}

Let $\m=(m_1,\dots,m_k)$ with $k\geq1$, not containing $\infty$, and set $M:=\lcm(\m)$. We will define a \textit{scrape} of $\m$, which is a factor of $\m$ maximal with respect to a certain property. Scrapes are \textit{almost} the correct notion of the factors needed in Theorem \ref{thm:distinguishing_factors}; however, when $2$ or $3$ divides $M$, we often will require a slightly stronger notion of the \textit{closure} of a scrape. In this section, we will denote multisets with condensed notation either by $$\ds \m=\bigoplus_{i=1}^{k}m_i,\text{ or by }\ds \m=\bigoplus_{d\in\m}d.$$ 

\begin{defn}\mbox{}
    \begin{enumerate}
        \item Given $s,d\mid M$, define $d_s:=\gcd(d,M/s)$, called the \textit{$s$-scrape} of an element $d$.
    
        \item The \emph{$s$-scrape} of $\m$ is given by element-wise scrapes $\ds \m_s:=\bigoplus_{d\in\m}d_s=\bigoplus_{d\in\m}\gcd(d,M/s)$. This is equivalent to having $\ds{\m_s=\max_{\cvec\mid\m}\{\lcm(\cvec)=M/s\}}$, i.e., if $\cvec\mid\m$ such that $\lcm(\cvec)=M/s$, then $\cvec\mid\m_s$. For the contravariant version of a scrape, we define $\ds \m^t:=\m_{M/t}=\max_{\x\mid\m}\{\lcm(\x)=t\}$.
        
        \item  We say that a factor $\cvec\mid\m$ is a \emph{scrape} of $\m$, if there is some $s\mid M$ such that $\cvec=\m_s$.
    \end{enumerate}
\end{defn}

The following properties of scrapes are straightforward to verify.
\begin{prop}\label{prop:scrape_properties}\mbox{}
    \begin{enumerate}
        \item  There is an order-preserving bijection between the scrapes of $\m$ and the factors of $M$, via the $\lcm$ operator. In particular, $\cvec\mid\dvec\iff \lcm(\cvec)\mid\lcm(\dvec)$.
        \item Incremental scrapes: Suppose $\ell\mid s$ and $\ell$ is prime. If we set $\m_s=(s_1,\dots,s_k)$, then $\m_{s\ell}:=(x_1,\dots,x_k)$, where $$x_i=\begin{cases}
            s_i &\text{if } v_{\ell}(s_i)<v_{\ell}(s)\\
            s_i/\ell &\text{if } v_{\ell}(s_i)=v_{\ell}(s).
        \end{cases}$$
        \item Reversed increment: Suppose $s\ell\mid M$, for a prime $\ell$. If we set $\m_s=(s_1,\dots,s_k)$, then $\m_{s/\ell}:=(y_1,\dots,y_k)$, where $$y_i=\begin{cases}
            s_i &\text{if } v_{\ell}(s_i)=v_{\ell}(m_i)\\
            s_i\ell &\text{if } v_{\ell}(s_i)<v_{\ell}(m_i).
        \end{cases}$$
        \item Transitivity of scrapes: Suppose $st\mid M$, then $(\m_s)_t=(\m_t)_s=\m_{st}$.
        \item Scrapes preserve $\med_i$'s: Suppose $|\m|=|\n|$ and $\med_i(\m)=\med_i(\n)$, for each $1\leq i\leq k$. Then for each $s\mid M$, we have $\med_i(\m_s)=\med_i(\n_s)$.
    \end{enumerate}
\end{prop}

We now prove that scrapes are semi-multiplicative functions.

\begin{propn}\label{prop:scrapes_are_semimultiplicative}\mbox{}
    \begin{enumerate}
        \item Suppose $ab\mid M$ and $\gcd(a,b)=1$. Then for each $d\mid M$, we have $d\cdot d_{ab}=d_a d_b$.
        
        \item Suppose $a_1a_2\dots a_l\mid M$ and $\gcd(a_i,a_j)=1$, for each $i\neq j$. Then for every $d\mid M$, we have $$d^{l-1}\cdot d_{a_1a_2\dots a_l}=d_{a_1}d_{a_2}\dots d_{a_l}.$$
    \end{enumerate}
\end{propn}

\begin{proof}
    The first statement follows since notice $$(ab)\cdot\gcd(M/a,M/b)=\gcd(bM,\ aM)=M\cdot\gcd(a,b)=M,\text{ which implies }\gcd(M/a,M/b)=M/(ab).$$ Hence, we have \begin{align*}
        d_ad_b=\gcd(d,M/a)\cdot\gcd(d,M/b) &= \gcd(d^2,\ dM/a,\ dM/b,\ M^2/(ab))\\
        &= \gcd(d^2,\ dM/(ab),\ M^2/(ab))\\
        &= \gcd(d^2,\ dM/(ab))\\
        &= d\cdot\gcd(d,\ M/(ab))\\
        &= d\cdot d_{ab}.
    \end{align*} The second statement follows inductively.
\end{proof}

We now define the \textit{closure} of a factor $\cvec\mid\m$; however, we must fix an ordering for the elements in $\cvec$ relative to $\m$. A scrape $\m_s\mid\m$, in particular, already comes with such a fixed ordering.

\begin{defn}\label{def:closure}
    Let $\cvec\mid\m$ be a factor, where the elements of $\cvec=(c_1,\dots,c_k)$ and $\m=(m_1,\dots,m_k)$ are ordered in such a way so that $c_i\mid m_i$. With respect to this chosen order, define the \textit{closure} $\ol{\cvec}$ of $\cvec$, where the elements $c_i'$ of $\ol{\cvec}=(c_1',\dots,c_k')$ are given by $$c_i':=\begin{cases}
        2 & \text{if $c_i=1$ and $2\mid m_i$, $3\nmid m_i$}\\
        3 & \text{if $c_i=1$ and $3\mid m_i$}\\
        3 & \text{if $c_i=2$ and $2,3\mid m_i$}\\
        c_i & \text{otherwise}.
    \end{cases}$$
\end{defn}

 The closure $\ol{\cvec}$ is a factor of $\m$ and is non-trivial only when $2\mid M$ or $3\mid M$. Observe that $\chi(\ol{\cvec})\leq\chi(\cvec)$. The main purpose of the closure $\ol{\cvec}$ is for results involving maximally smooth representations to $\PSL(2,q)$.

\begin{propn}
    If $q>1$ is an odd prime power with a $\cvec$-smooth representation $\vp\colon\Delta(\m)\to\PSL(2,q)$, then there is also some $\ol{\cvec}$-smooth representation $\vp'\colon\Delta(\m)\to\PSL(2,q)$. In particular, if $\ol{\cvec}\neq\cvec$, then there can never be a $\cvec$-maximally smooth representation to $\PSL(2,q)$.
\end{propn}

\begin{proof}
    If $\ol{\cvec}\neq\cvec$, then $\chi(\ol{\cvec})<\chi(\cvec)$. The result follows from Theorem \ref{thm:macbeath} along with the observation that elements of orders 2 and 3 exist in every $\PSL(2,q)$.
\end{proof}

The next two subsections aim to prove the following theorem.

\begin{thm}\label{thm:distinguishing_scrape_closures}
    Let $\m\neq\n$ have the same length, neither of which contains 1 nor $\infty$. Suppose $\De(\m)^{\Ab}\cong\De(\n)^{\Ab}$, and $\chi(\m)=\chi(\n)$. Set $M:=\lcm(\m)=\lcm(\n)$, then there exists $s\mid M$ such that $\chi(\ol{\m_s})\neq\chi(\ol{\n_s})$.
\end{thm}

\subsection{Euler characteristics and homogeneous linear systems}

Suppose $\m\neq\n$ have the same length, neither of which contains 1 nor $\infty$, and suppose $\De(\m)^{\Ab}\cong\De(\n)^{\Ab}$, $\chi(\m)=\chi(\n)$, and set $M:=\lcm(\m)=\lcm(\n)$. There is no ambiguity in the definition of an \textit{$s$-scrape} of $d\in\m\cup\n$, i.e., $d_s:=\gcd(d,M/s)$. Before we prove Theorem \ref{thm:distinguishing_scrape_closures}, we will prove a weaker statement, one which does not take closures into account.

\begin{thm}\label{thm:distinguishing_scrapes}
    There exists $s\mid M$ such that $\chi(\m_s)\neq\chi(\n_s)$.
\end{thm}

We used the computational algebra system GAP \cite{gap} to help facilitate our findings.

The rest of this subsection is for proving the contrapositive of Theorem \ref{thm:distinguishing_scrapes}. We start by assuming that $\chi(\m_s)=\chi(\n_s)$ for every $s\mid M$. We set $$\m:=\bigoplus_{d\mid M}d^{(A_d)}\text{ and }\n:=\bigoplus_{d\mid M}d^{(B_d)},$$ where $a^{(i)}$ denotes the cone order $a$ repeated $i$ times. It remains to show that $\m=\n$, which is equivalent to proving that $\De_d:=A_d-B_d=0$ for each $d\mid M$. Observe for each $s\mid M$, \begin{equation}
    \chi(\m_s)=\chi(\n_s) \iff \sum_{d\mid M}\frac{A_d}{d_s}=\sum_{d\mid M}\frac{B_d}{d_s} \iff \sum_{d\mid M}\frac{1}{d_s}\De_d=0,\label{eqn:homogeneous_eqn}
\end{equation} which is a homogeneous system of linear equations in the variables $\De_d$, for $d\mid M$. Let $E=(e_{s,d})$ be the square matrix representing (\ref{eqn:homogeneous_eqn}), with rows indexed by each $s\mid M$ and columns indexed by each $d\mid M$, i.e., $e_{s,d}=1/d_s$. We will have proved Theorem \ref{thm:distinguishing_scrapes} once we argue that $E$ has full rank. To this end, apply to $E$ the following recursively defined row operations to get the matrix $X=(x_{s,d})$, and denote $E_s$ and $X_s$ as the $s$-th rows of $E$ and $X$, respectively. \begin{enumerate}
    \item Set the row $X_1:=E_1$.
    \item For each $s\mid M$ with $s>1$, take $\ds X_s:=E_s-\sum_{\substack{c\mid s\\ c\neq s}}X_c$.
\end{enumerate} Therefore, with $d\mid M$ fixed we have the following formula as a function of $s$: \begin{equation}\label{eqn:scrape_sum_formula}
    \ds \frac{1}{d_s}=\sum_{c\mid s}x_{c,d}.
\end{equation} Define the \textit{Möbius function}, where given an integer $n>0$, $$\mu(n):=\begin{cases}
    1 & \text{if }n=1\\
    (-1)^i & \text{if $n$ is the product of $i$ distinct primes}\\
    0 & \text{if $n$ has any repeated prime factors}.
\end{cases}$$ The Möbius function satisfies the following properties: \begin{enumerate}
    \item \textit{Sum of the Möbius function over divisors:} $\ds \sum_{d\mid n}\mu(d)=\begin{cases}
        1 & \text{if }n=1\\
        0 & \text{if }n>1
    \end{cases}$

    \item \textit{Möbius function is multiplicative:} For $a,b>0$ coprime, $\mu(ab)=\mu(a)\mu(b)$.

    \item \textit{Möbius inversion formula:} Suppose $f,g$ are arithmetic functions such that $\ds g(n)=\sum_{d\mid n}f(d)$, for each $n\geq1$. Then $\ds f(n)=\sum_{d\mid n}\mu(n/d)g(d)$, for each $n\geq1$.
\end{enumerate} By applying the Möbius inversion formula to (\ref{eqn:scrape_sum_formula}), we can express the entries of $X$ in the variable $s$ with the following sum \begin{equation}
    x_{s,d}=\sum_{c\mid s}\frac{\mu(s/c)}{d_c}.\label{eqn:sum_formula_for_X}
\end{equation}

Our aim is to express the entries $x_{s,d}$ of $X$ by a closed formula equivalent to (\ref{eqn:sum_formula_for_X}).

\begin{propn}\label{prop:entries_X_are_semimultiplicative}\mbox{}
    \begin{enumerate}
        \item Suppose $ab\mid M$ and $\gcd(a,b)=1$. Then for each $d\mid M$, we have $x_{a,d}\cdot x_{b,d}=\frac{1}{d}\cdot x_{ab,d}$.

        \item Suppose $a_1a_2\dots a_l\mid M$ and $\gcd(a_i,a_j)=1$, for each $i\neq j$. Then for each $d\mid M$, we have $$x_{a_1,d}\cdot x_{a_2,d}\cdot\ldots\cdot x_{a_l,d}=\frac{1}{d^{l-1}}\cdot x_{a_1a_2\dots a_l,d}.$$ 
    \end{enumerate}
\end{propn}

\begin{proof}
    The first statement follows from equation (\ref{eqn:sum_formula_for_X}), Proposition \ref{prop:scrapes_are_semimultiplicative}, and basic Möbius function properties: $$x_{a,d}\cdot x_{b,d}=\left(\sum_{f\mid a}\frac{\mu(a/f)}{d_f}\right)\left(\sum_{g\mid b}\frac{\mu(b/g)}{d_g}\right)=\sum_{f\mid a}\sum_{g\mid b}\frac{\mu(a/f)\mu(b/g)}{d_f d_g}=\frac{1}{d}\sum_{fg\mid ab}\frac{\mu(ab/fg)}{d_{fg}}=\frac{1}{d}\cdot x_{ab,d}.$$ The second statement follows inductively.
\end{proof}

\begin{thm}
    Fix $d\mid M$. Let $\ell$ be a prime and suppose that $\ell^a\mid\mid M$, for some $a>0$. For $0<j\leq a$, we have $$x_{\ell^j,d}=0\iff \ell^{a-j+1}\nmid d.$$ Moreover, whenever $\ell^{a-j+1}\mid d$ so that for some $k\geq1$, $\ell^{a-j+k}\mid\mid d$, i.e., $d_{\ell^j}=d/\ell^k$, we have \begin{equation}
        x_{\ell^j,d}=\frac{1}{d}\phi(\ell^k),\label{eqn:X_has_totient_fct}
    \end{equation} where $\phi$ denotes the \textit{Euler totient} function.
\end{thm}

\begin{proof}
    Suppose $\ell^{a-j+1}\nmid d$. By definition, we have $d_{\ell^i}=d$ for each $0\leq i\leq j$. Therefore by (\ref{eqn:sum_formula_for_X}), we have $$\ds x_{\ell^j,d}=\sum_{i=0}^j\frac{\mu(\ell^i)}{d_{\ell^{j-i}}}=\frac{1}{d}\sum_{i=0}^j\mu(\ell^i)=\frac{1}{d}\sum_{c\mid\ell^j}\mu(c)=0,$$ since $j>0$. Now suppose $\ell^{a-j+1}\mid d$, and in particular, set $1\leq k\leq j$ so that it satisfies $\ell^{a-j+k}\mid\mid d$. We use the Möbius expression (\ref{eqn:sum_formula_for_X}) to yield the following: $$\ds x_{\ell^j,d}=\sum_{i=0}^j\frac{\mu(\ell^i)}{d_{\ell^{j-i}}}=\sum_{i=0}^{k-1}\frac{\mu(\ell^i)}{d_{\ell^{j-i}}}+\sum_{i=k}^{j}\frac{\mu(\ell^i)}{d_{\ell^{j-i}}}=\frac{1}{d}\left(\sum_{i=0}^{k-1}\ell^{k-i}\mu(\ell^i)+\sum_{i=k}^{j}\mu(\ell^i)\right)\neq0,$$ since we have $$\sum_{i=0}^{k-1}\ell^{k-i}\mu(\ell^i)=\begin{cases}
        \ell^k-\ell^{k-1} & \text{ if }k>1\\
        \ell & \text{ if }k=1
    \end{cases}$$ and $$\sum_{i=k}^{j}\mu(\ell^i)=\begin{cases}
        0 & \text{ if }k>1\\
        -1 & \text{ if }k=1.
    \end{cases}$$
\end{proof}

\begin{cor}\label{cor:entrywiseCriterion}
    Suppose $d\mid M$, with prime factorization $M=p_1^{a_1}p_2^{a_2}\dots p_q^{a_q}\dots p_r^{a_r}$. For $0<j_t\leq a_t$ where $1\leq t\leq q$, and setting $s=p_1^{j_1}p_2^{j_2}\dots p_q^{j_q}\neq1$, we have $$x_{s,d}=x_{(p_1^{j_1}p_2^{j_2}\dots p_q^{j_q}),\ d}=0\iff\text{there is some $1\leq t\leq q$ satisfying } p_t^{a_t-j_t+1}\nmid d.$$

    In the case where each $1\leq t\leq q$ satisfies $p_t^{a_t-j_t+1}\mid d$, and in particular suppose $p_t^{a_t-j_t+k_t}\mid\mid d$ for $1\leq k_t\leq j_t$, then \begin{equation}\label{eqn:formula_for_X}
        x_{s,d}=\frac{1}{d}\phi(p_1^{k_1}p_2^{k_2}\dots p_q^{k_q})=\frac{\phi(d/d_s)}{d}.
    \end{equation}
\end{cor}

\begin{proof}
    This follows from Proposition \ref{prop:entries_X_are_semimultiplicative}, equation (\ref{eqn:X_has_totient_fct}), and the multiplicativity of the Euler totient function.
\end{proof}

We set up convenient notation from the corollary.

\begin{defn}
    Let $s\mid M$ and suppose $M=p_1^{a_1}p_2^{a_2}\dots p_q^{a_q}\dots p_r^{a_r}$ and $s=p_1^{j_1}p_2^{j_2}\dots p_q^{j_q}$ are their prime factorizations with $0<j_t\leq a_t$. We define the \textit{pivot} of $s$ by the following: define $\tilde{1}:=1$, and for $s\neq1$ set $\tilde{s}:=p_1^{a_1-j_1+1}p_2^{a_2-j_2+1}\dots p_q^{a_q-j_q+1}$.
\end{defn}

 Corollary \ref{cor:entrywiseCriterion} can be succinctly written as $x_{s,d}=\begin{cases}
    0 & \text{ if }\tilde{s}\nmid d\\
    \frac{\phi(d/d_s)}{d} & \text{ if }\tilde{s}\mid d.
\end{cases}$. Also, observe that $\tilde{\tilde{s}}=s$, for each $s\mid M$.

\begin{prop}\label{prop:E_has_full_rank}
    For each column $d\mid M$, the row $\tilde{d}\mid M$ is the unique row in satisfying the following property: $x_{\tilde{d},d}\neq0$ and $x_{\tilde{d},d'}=0$ for each $d'<d$. Equivalently, there is some permutation of the rows of $X$ which is an upper triangular matrix and each entry $x_{\tilde{d},d}\neq0$ is a pivot entry of $X$.
\end{prop}

\begin{proof}
    Given $d\mid M$, suppose $x_{s,d}\neq0$ and $x_{s,d'}=0$ for each $d'<d$. By Corollary \ref{cor:entrywiseCriterion}, we have $\tilde{s}\mid d$ but $\tilde{s}\nmid c$, for any proper divisor $c$ of $d$. It must be that $\tilde{s}=d$, that is, $s=\tilde{d}$.
\end{proof}

This proves Theorem \ref{thm:distinguishing_scrapes} since Proposition \ref{prop:E_has_full_rank} establishes that the matrix $E$ has full rank.

\subsection{Proof of Theorem \ref{thm:distinguishing_scrape_closures}}

In this section, we prove the contrapositive of Theorem \ref{thm:distinguishing_scrape_closures}. Similar to the previous subsection, we will produce another homogeneous system of linear equations, corresponding to a matrix $Y$. The initial aspects to this proof resemble that of Theorem \ref{thm:distinguishing_scrapes} for the matrix $X$, but the differences from $X$ will help to characterize $Y$.

We start by assuming that $\chi(\ol{\m_s})=\chi(\ol{\n_s})$ for every $s\mid M$, and it remains to show that $\m=\n$, i.e., $\De_d:=A_d-B_d=0$ for each $d\mid M$.

Consider the closure $\ol{\cvec}$ of $\cvec\mid\m$. For each element $z\in\cvec$, we denote $z'$ as the corresponding element in $\ol{\cvec}$, according to Definition \ref{def:closure}. Observe the only 3 possibilities where $z\neq z'$ could occur: \begin{itemize}
    \item $z=1$ and $z'=2$,
    \item $z=1$ and $z'=3$, and
    \item $z=2$ and $z'=3$.
\end{itemize} For each $s\mid M$, \begin{equation}
    \chi(\ol{\m_s})=\chi(\ol{\n_s}) \iff \sum_{d\mid M}\frac{1}{d_s'}\De_d=0,\label{eqn:homogeneous_eqn2}
\end{equation} and let $F=(f_{s,d})$ be the square matrix representing equation (\ref{eqn:homogeneous_eqn2}), i.e., $f_{s,d}=1/d_s'$. This time; however, we will see that $F$ might not have full rank. In fact, we will later append 3 more equations to $F$ to ensure a full rank matrix. At this point, we would like to understand the nullity of $F$, and we apply recursively defined row operations to get the matrix $Y=(y_{s,d})$, and denote $F_s$ and $Y_s$ as the $s$-th rows of $F$ and $Y$, respectively. \begin{enumerate}
    \item Set the row $Y_1:=F_1$.
    \item For each $s\mid M$ with $s>1$, take $\ds Y_s:=F_s-\sum_{\substack{c\mid s\\ c\neq s}}Y_c$.
\end{enumerate} Therefore, with $d\mid M$ fixed we have the following formula as a function of $s$: \begin{equation}\label{eqn:scrape_closure_sum_formula}
    \ds \frac{1}{d_s'}=\sum_{c\mid s}y_{c,d}.
\end{equation} By applying the Möbius inversion formula to (\ref{eqn:scrape_closure_sum_formula}), we can express the entries of $Y$ by \begin{equation}
    y_{s,d}=\sum_{c\mid s}\frac{\mu(s/c)}{d_c'}.\label{eqn:sum_formula_for_Y}
\end{equation} The matrices $X$ and $Y$ are not very different from each other. We compare the matrices $X$ and $Y$, taking advantage that $X$ is a square matrix with full rank. We use (\ref{eqn:sum_formula_for_X}) and (\ref{eqn:sum_formula_for_Y}) to consider the difference \begin{align*}
    x_{s,d}-y_{s,d}= \sum_{\substack{c\mid s\\ d_c\neq d_c'}}\left(\frac{1}{d_c}-\frac{1}{d_c'}\right)\mu(s/c)= \underbrace{\frac{1}{2}\sum_{\substack{c\mid s\\ d_c=1\\ d_c'=2}}\mu(s/c)}_{R}+\underbrace{\frac{2}{3}\sum_{\substack{c\mid s\\ d_c=1\\ d_c'=3}}\mu(s/c)}_{S}+\underbrace{\frac{1}{6}\sum_{\substack{c\mid s\\ d_c=2\\ d_c'=3}}\mu(s/c)}_{T}.
\end{align*}

Take $M=2^a3^bp_1^{a_1}\dots p_q^{a_q}\dots p_r^{a_r}$ to be a product of distinct prime powers, where $a,b\geq 0$ and $a_t>0$. In the following scenarios, we exploit the definitions of scrapes and closures to characterize the values of $s,d\mid M$ which contribute (or contribute zero) to the summands $R$, $S$, and $T$.

\begin{enumerate}
    \item[$R$:] The condition $d_c=1$ and $d_c'=2$ is equivalent to having $d=2^jp_1^{j_1}\dots p_q^{j_q}$, where $j,j_t>0$, and\\ $2^ap_1^{a_1}\dots p_q^{a_q}\mid c$. Since $c\mid s$, we set $s=2^ap_1^{a_1}\dots p_q^{a_q}k$ so that $c=2^ap_1^{a_1}\dots p_q^{a_q}k'$, for $k'\mid k$. Hence, the summand $R$ simplifies to $$\frac{1}{2}\sum_{\substack{c\mid s\\ d_c=1\\ d_c'=2}}\mu(s/c)=\frac{1}{2}\sum_{k'\mid k}\mu(k/k')=\begin{cases}
        1/2, & \text{if }k=1\\
        0, & \text{if }k>1.
    \end{cases}$$

    \item[$S$:] The condition $d_c=1$ and $d_c'=3$ is equivalent to having $d=3^lp_1^{j_1}\dots p_q^{j_q}$, where $l,j_t>0$ and for the convenience of notation, we allow for the possibility for $2\in\{p_1,\dots,p_q\}$. Additionally, $3^bp_1^{a_1}\dots p_q^{a_q}\mid c$. Since $c\mid s$, we set $s=3^bp_1^{a_1}\dots p_q^{a_q}k$ so that $c=3^bp_1^{a_1}\dots p_q^{a_q}k'$, for $k'\mid k$. Hence, the summand $S$ simplifies to $$\frac{2}{3}\sum_{\substack{c\mid s\\ d_c=1\\ d_c'=3}}\mu(s/c)=\frac{2}{3}\sum_{k'\mid k}\mu(k/k')=\begin{cases}
        2/3, & \text{if }k=1\\
        0, & \text{if }k>1.
    \end{cases}$$
    
    \item[$T$:] We characterize the conditions $d_c=2$ and $d_c'=3$ using two cases for $d$: the cases $4 \mid d$ and $2\mid\mid d$.\\ If $4 \mid d$, the conditions $d_c=2$ and $d_c'=3$ are equivalent to having $d=2^j3^lp_1^{j_1}\dots p_q^{j_q}$, where $l,j_t>0$ and $j\geq2$; $2^{a-1}3^bp_1^{a_1}\dots p_q^{a_q}\mid c$; and $2^{a-1}\mid\mid c$. Since $c\mid s$, it must be that $s=2^{a-1}3^bp_1^{a_1}\dots p_q^{a_q}k$ or $s=2^{a}3^bp_1^{a_1}\dots p_q^{a_q}k$.
    
    \begin{enumerate}
        \item In the case $2^{a-1}\mid\mid s$, suppose $s=2^{a-1}3^bp_1^{a_1}\dots p_q^{a_q}k$ and $c=2^{a-1}3^bp_1^{a_1}\dots p_q^{a_q}k'$, for some $k'\mid k$. The summand $T$ simplifies to $$\frac{1}{6}\sum_{\substack{c\mid s\\ d_c=2\\ d_c'=3}}\mu(s/c)=\frac{1}{6}\sum_{k'\mid k}\mu(k/k')=\begin{cases}
        1/6, & \text{if }k=1\\
        0, & \text{if }k>1.
    \end{cases}$$

        \item In the case $2^{a}\mid\mid s$, suppose $s=2^{a}3^bp_1^{a_1}\dots p_q^{a_q}k$ and $c=2^{a-1}3^bp_1^{a_1}\dots p_q^{a_q}k'$, for some $k'\mid k$. The summand $T$ simplifies to $$\frac{1}{6}\sum_{\substack{c\mid s\\ d_c=2\\ d_c'=3}}\mu(s/c)=\frac{1}{6}\sum_{k'\mid k}\mu(2k/k')=\frac{1}{6}\sum_{k'\mid k}-\mu(k/k')=\begin{cases}
        -1/6, & \text{if }k=1\\
        0, & \text{if }k>1.
    \end{cases}$$
    \end{enumerate}
    
    On the other hand if $2 \mid\mid d$, the conditions $d_c=2$ and $d_c'=3$ are equivalent to having $d=2\cdot3^lp_1^{j_1}\dots p_q^{j_q}$, where $l,j_t>0$; $3^bp_1^{a_1}\dots p_q^{a_q}\mid c$; and $2^{a}\nmid c$. Since $c\mid s$, it must be that $s=2^{j}3^bp_1^{a_1}\dots p_q^{a_q}k$ for some $j \geq0$.
    
    \begin{enumerate}
        \item If $2^{a}\nmid s$, suppose $s=2^{j}3^bp_1^{a_1}\dots p_q^{a_q}k$ and $c= 2^{i}3^bp_1^{a_1}\dots p_q^{a_q}k'$, for $j<a$, $k'\mid k$, and $\gcd(2,k)=1$. The summand $T$ simplifies to \begin{align*}
            \frac{1}{6}\sum_{\substack{c\mid s\\ d_c=2\\ d_c'=3}}\mu(s/c) &= \frac{1}{6}\sum_{i=0}^j\sum_{k'\mid k}\mu(2^{j-i}k/k')=\frac{1}{6}\left(\sum_{i=0}^j\mu(2^{j-i})\right)\left(\sum_{k'\mid k}\mu(k/k')\right)\\
            &=\frac{1}{6}\left(\begin{cases}
                1, & \text{if }j=0\\
                0, & \text{if }j>0
            \end{cases}\right)\left(\begin{cases}
                1, & \text{if }k=1\\
                0, & \text{if }k>0
            \end{cases}\right)=\begin{cases}
                1/6, & \text{if $(j,k)=(0,1)$}\\
                0, & \text{else}.
            \end{cases}
        \end{align*}
        
        \item If $2^{a}\mid\mid s$, suppose $s=2^{a}3^bp_1^{a_1}\dots p_q^{a_q}k$ and $c=2^{i}3^bp_1^{a_1}\dots p_q^{a_q}k'$, for $k'\mid k$. Recall that we must have $i<a$. The summand $T$ simplifies to \begin{align*}
            \frac{1}{6}\sum_{\substack{c\mid s\\ d_c=2\\ d_c'=3}}\mu(s/c)=\frac{1}{6}\sum_{i=0}^{a-1}\sum_{k'\mid k}\mu(2^{a-i}k/k') &= \frac{1}{6}\sum_{i=0}^a\sum_{k'\mid k}\mu(2^{a-i}k/k')-\frac{1}{6}\sum_{k'\mid k}\mu(2^{a-a}k/k')\\
            &=\frac{1}{6}\left(\sum_{i=0}^a\mu(2^{a-i})\right)\left(\sum_{k'\mid k}\mu(k/k')\right)-\frac{1}{6}\sum_{k'\mid k}\mu(k/k')\\
            &=\left(\begin{cases}
                1/6, & \text{if $(a,k)=(0,1)$}\\
                0, & \text{else}
            \end{cases}\right)-\left(\begin{cases}
                1/6, & \text{if $k=1$}\\
                0, & \text{if $k>1$}
            \end{cases}\right)\\
            &=\begin{cases}
                -1/6, & \text{if $k=1$ and $a > 0$}\\
                0, & \text{else}.
            \end{cases}
        \end{align*}
    \end{enumerate}
\end{enumerate}

Given the prime factorization $M=2^a3^bp_1^{a_1}\dots p_q^{a_q}\dots p_r^{a_r}$, where $a,b\geq0$ and $a_t>0$, the following table summarizes the range of differences $x_{s,d}-y_{s,d}$, for each $s,d\mid M$, in mutually exclusive cases. We assume that each of the powers of the primes for $s$ and $d$ appearing in the table below are positive.

\begin{center}
\begin{tabular}{ | l | l | c | c | c | c | } 
  \hline
  $s$ & $d$ & $x_{s,d}-y_{s,d}$ & $R+S+T$ & Does $\tilde{s}\mid d$? & Comments \\ 
  \hline
  $2^ap_1^{a_1}\dots p_q^{a_q}$ & $2^jp_1^{j_1}\dots p_q^{j_q}$ & 1/2 & $1/2+0+0$ & Yes & $2\mid d,\ 3\nmid d$\\ 
  \hline
  $3^bp_1^{a_1}\dots p_q^{a_q}$ & $3^lp_1^{j_1}\dots p_q^{j_q}$ & 2/3 & $0+2/3+0$ & Yes & $3\mid d,\ 2\nmid d$\\ 
  \hline
  $2^{a-1}3^bp_1^{a_1}\dots p_q^{a_q}$ & $2^j3^lp_1^{j_1}\dots p_q^{j_q}$ & 1/6 & $0+0+1/6$ & Yes & $4\mid d,\ 3\mid d,\ 2^{a-1}\mid\mid s$
  \\ 
  \hline
  $3^bp_1^{a_1}\dots p_q^{a_q}$ & $2\cdot3^lp_1^{j_1}\dots p_q^{j_q}$ &  1/6 & $0+0+1/6$ & Yes & $2\mid\mid d,\ 3\mid d,\ 2\nmid s$\\
  \hline
  $2^{a}3^bp_1^{a_1}\dots p_q^{a_q}$ & $2^j3^lp_1^{j_1}\dots p_q^{j_q}$ & 1/2 & $0+2/3-1/6$ & Yes & $2\mid d$, $3\mid d$, $2^a\mid\mid s$\\
  \hline
  else & else & 0 & $0+0+0$ & Irrelevant & $x_{s,d}=y_{s,d}$\\ 
  \hline
\end{tabular}
\end{center}

From Corollary \ref{cor:entrywiseCriterion}, the entries of the matrix $X$ are $x_{s,d}=\begin{cases}
    0 & \text{ if }\tilde{s}\nmid d\\
    \frac{\phi(d/d_s)}{d} & \text{ if }\tilde{s}\mid d.
\end{cases}$ Moreover, the table shows that $x_{s,d}=0$ implies $y_{s,d}=0$. Hence, just like for $X$, the matrix $Y$ is an upper triangular after applying the exact same row interchanges needed for $X$. Furthermore, recall that the matrices $X,Y$ are relative to the value of $M$, and the table above proves the following: 

\begin{lemma}\label{lem:independence_of_sd}
    The values of $x_{s,d}$ and $y_{s,d}$ are independent of the multiple $M$ chosen for $s$ and $d$.
\end{lemma}

We now show the nullity of $F$ (and $Y$) is at most 3 by comparing the rows $X_s$ and $Y_s$. From Theorem \ref{thm:distinguishing_scrapes}, we have that $X$ has full rank; in particular, for each column $d\mid M$, the entry $x_{\tilde{d},d}\neq0$ in $X$ is a pivot. We prove that $y_{\tilde{d},d}=0$ is not a pivot entry of $Y$ if and only if $d\in\{2,3,12\}$. We prove this by analyzing each table entry, paying particular attention to where $x_{\tilde{d},d}\neq y_{\tilde{d},d}$:

\begin{enumerate}
    \item When $s=2^ap_1^{a_1}\dots p_q^{a_q}$ and $d=2^jp_1^{j_1}\dots p_q^{j_q}$, we have $\tilde{s}\mid d$ and $$x_{s,d}=\frac{\phi(d/d_s)}{d}=\frac{\phi(d)}{d}=\frac{\phi(2^jp_1^{j_1}\dots p_q^{j_q})}{2^jp_1^{j_1}\dots p_q^{j_q}}=\frac{(p_1-1)\dots(p_q-1)}{2p_1\dots p_q},$$ which is equal to 1/2 if and only if $q=0$. Thus, $y_{s,d}=x_{s,d}-1/2=0$ if and only if $s=2^a$ and $d=2^j$ with $j>0$. Therefore, the matrix $Y$ does not have a pivot at $y_{2^a,\widetilde{2^a}}=y_{\widetilde{2},2}=0$.

    \item When $s=3^bp_1^{a_1}\dots p_q^{a_q}$ and $d=3^lp_1^{j_1}\dots p_q^{j_q}$, we have $\tilde{s}\mid d$ and $$x_{s,d}=\frac{\phi(d/d_s)}{d}=\frac{\phi(d)}{d}=\frac{\phi(3^lp_1^{j_1}\dots p_q^{j_q})}{3^lp_1^{j_1}\dots p_q^{j_q}}=\frac{2(p_1-1)\dots(p_q-1)}{3p_1\dots p_q},$$ which is equal to 2/3 if and only if $q=0$. Thus, $y_{s,d}=x_{s,d}-2/3=0$ if and only if $s=3^b$ and $d=3^l$ with $l>0$. Therefore, the matrix $Y$ does not have a pivot at $y_{3^b,\widetilde{3^b}}=y_{\widetilde{3},3}=0$.

    \item When $s=2^{a-1}3^bp_1^{a_1}\dots p_q^{a_q}$ and $d=2^j3^lp_1^{j_1}\dots p_q^{j_q}$ with $j\geq2$, we have $$x_{s,d}=\frac{\phi(d/d_s)}{d}=\frac{\phi(d/2)}{d}=\frac{\phi(2^{j-1}3^lp_1^{j_1}\dots p_q^{j_q})}{2^j3^lp_1^{j_1}\dots p_q^{j_q}}=\frac{(p_1-1)\dots(p_q-1)}{6p_1\dots p_q},$$ which is equal to 1/6 if and only if $q=0$. Thus, $y_{s,d}=x_{s,d}-1/6=0$ if and only if $s=2^{a-1}3^b$ and $d=2^j3^l$ with $j\geq2$ and $l>0$. Therefore, the matrix $Y$ does not have a pivot at $y_{2^{a-1}3^b,\widetilde{2^{a-1}3^b}}=y_{\widetilde{12},12}=0$.

    \item When $s=3^bp_1^{a_1}\dots p_q^{a_q}$ and $d=2\cdot3^lp_1^{j_1}\dots p_q^{j_q}$, we have $$x_{s,d}=\frac{\phi(d/d_s)}{d}=\frac{\phi(d/2)}{d}=\frac{\phi(3^lp_1^{j_1}\dots p_q^{j_q})}{2\cdot3^lp_1^{j_1}\dots p_q^{j_q}}=\frac{(p_1-1)\dots(p_q-1)}{6p_1\dots p_q},$$ which is equal to 1/6 if and only if $q=0$. Thus, $y_{s,d}=x_{s,d}-1/6=0$ if and only if $s=3^b$ and $d=2\cdot3^l$. Notice (2) already addresses the row $s=3^b$, so this case can be ignored for our nullity count.
    
    \item When $s=2^{a}3^bp_1^{a_1}\dots p_q^{a_q}$ and $d=2^j3^lp_1^{j_1}\dots p_q^{j_q}$, we have $$x_{s,d}=\frac{\phi(d/d_s)}{d}=\frac{\phi(d)}{d}=\frac{\phi(2^j3^lp_1^{j_1}\dots p_q^{j_q})}{2^j3^lp_1^{j_1}\dots p_q^{j_q}}=\frac{(p_1-1)\dots(p_q-1)}{6p_1\dots p_q},$$ which can never be equal to 1/2. Thus, $y_{s,d}=x_{s,d}-1/2\neq0$ for any such choices of $s,d$.

    \item For all the other columns $d$, we have $x_{s,d}=y_{s,d}$, so the pivots in these columns remain unchanged from $X$ to $Y$.
\end{enumerate}

This proves that the nullity of $F$ (and $Y$) is at most three, exactly in the columns $d=2,3$ and $12$, whichever ones exist in $Y$. Now, we append to both $F$ and $Y$ three specific linear equations in the variables $\De_d$, for each $d\mid M$, and call the resulting matrices $F'$ and $Y'$, respectively. We prove that $F'$ (and $Y'$) has full rank:

\begin{enumerate}
    \item Regarding the lack of pivot in column $d=2$ (resp. $d=3$), we append the equation $$\ds \sum_{\substack{2\mid\mid d\\ d\mid M}}\De_d=0 \text{ (resp. }\sum_{\substack{3\mid\mid d\\ d\mid M}}\De_d=0),$$ which is a consequence of $\m$ and $\n$ having the same abelianization. Observe that this equation leads with $\De_2$ (resp. $\De_3$), hence appending this equation patches the lack of a pivot. Recycling notation, we call the appended row $Y_{\tilde{2}}$ (resp. $Y_{\tilde{3}}$).
    
    \item Regarding the lack of pivot in column $d=12$, append the equation $\ds\sum_{d\mid M}\Delta_d=0$, which is equivalent to having $\m$ and $\n$ be the same length. Recycling notation, denote this row by $Y_{\widetilde{12}}$. Let $D$ denote the set of all divisors of $M$ less than or equal to 12, and let $\tilde{D}=\{\tilde{d}\mid d\in D\}$. Since the matrix $Y$ is (up to row permutations) an upper triangular matrix we can apply row operations to $Y_{\widetilde{12}}$ using rows $Y_{\tilde{j}}$, where $j\in D\setminus\{12\}$. We may also focus our attention only on the submatrix of entries where $(s,d)\in \tilde{D}\times D$; by Lemma \ref{lem:independence_of_sd}, this submatrix is uniquely determined by the set $D$. By brute force, we were able to verify for all 32 cases, which depend on the divisibility of $M$, that $Y_{\widetilde{12}}$ has a pivot at its $d=12$ entry.
\end{enumerate}

Therefore, with the addition of these 3 equations to $F$ and to $Y$, we have that $F'$ and $Y'$ have full rank. This proves Theorem \ref{thm:distinguishing_scrape_closures}.

\subsection{Badness and goodness}

For each $s\mid M$, recall the contravariant definition of a scrape $\m^s:=\m_{M/s}$, so that $\lcm(\m^s)=s$. Also, let $\ol{\m}^s$ denote the closure of $\m^s$. For the subsection, suppose $\m\neq\n$ are both good.

\begin{lemma}\label{lem:good_distinguishing_scrapes}
    Let $g,h$ be distinct primes dividing $M$ such that $h^b\mid\mid M$. Suppose the following scrapes are bad \begin{align}
        \begin{split}\label{eqn:gh_scrapes}
            \m^{s} &= (g^ih^j,g^{i+1}h^b,1^{(k-2)})\\
            \n^{s} &= (g^ih^b,g^{i+1}h^j,1^{(k-2)}),
        \end{split}
    \end{align} for $s=g^{i+1}h^b$, $0\leq j<b$, and where $i\geq0$ is chosen as the largest such power of $g$ satisfying $\m^{g^{i}h^b}=\n^{g^{i}h^b}$. Then there exists a prime $\ell\not\in\{g,h\}$ such that either

    \begin{enumerate}
        \item $\chi(\ol{\m}^{s\ell})\neq\chi(\ol{\n}^{s\ell})$ and at least one of $\ol{m}^{s\ell},\ol{\n}^{s\ell}$ is good, or
        \item exactly one of $\ol{m}^{s\ell}$ and $\ol{\n}^{s\ell}$ is bad, and exactly one is good.
    \end{enumerate}
\end{lemma}

\begin{proof}
    Fix the order of the elements of $\m=(m_1,m_2,m_3,\dots,m_k)$ and $\n=(n_1,n_2,n_3,\dots,n_k)$ so that component-wise $\m^{s}\mid\m$ and $\n^{s}\mid\n$. There exists some prime $\ell\mid M$ which are distinct from $g,h$; if no such prime $\ell$ exists, then $\m$ and $\n$ would both either have to be bad or equal to each other, contrary to our assumption. In particular, we can assume, say, $\ell\mid m_3$. By Proposition \ref{prop:scrape_properties}, we can write \begin{align}\label{eqn:scrape_candidates}
        \begin{split}
            \m^{s\ell} &= (\ell^{\delta_1}g^ih^j,\ \ell^{\delta_2}g^{i+1}h^b,\ \ell^{\delta_3},\ \ell^{\delta_4},\ \x),\\
            \n^{s\ell} &= (\ell^{\eps_1}g^{i}h^b,\ \ell^{\eps_2}g^{i+1}h^j,\ \ell^{\eps_3},\ \ell^{\eps_4},\ \x),
        \end{split}
    \end{align} where $\x$ contains only the elements $1$ and $\ell$ and are common to each of the multisets. We also have that $\delta_q,\eps_r\in\{0,1\}$, $\sum\delta_q=\sum\eps_r$, and $\delta_3=1$ (since $\ell\mid m_3$). To this end, we used Mathematica \cite{mathematica} to run each of the viable cases for the values of $\{\delta_q,\eps_r\}$ to confirm symbolically that $\chi(\ol{m}^{s\ell})-\chi(\ol{n}^{s\ell})$ cannot be equal to zero. It remains to prove that at least one of $\m^{s\ell}$ or $\n^{s\ell}$ is good for some choice of $\ell\nmid s$. If $(i,j)\neq(0,0)$, then the goodness follows since we have that $\delta_3=1$. On the other hand, if $i=j=0$, the possibilities where both $\m^{s\ell}$ and $\n^{s\ell}$ are both bad occur precisely $$\m^{s\ell}=(1,\ \ell^{\delta_2}gh^b,\ \ell,\ 1^{(k-3)}),\quad \n^{s\ell}=(\ell^{\eps_1}h^b,\ \ell^{\eps_2}g,\ 1,\ 1^{(k-3)}),$$ where $\delta_2+1=\eps_1+\eps_2$. Suppose for each choice of prime $\ell\nmid s$, we have that $\n^{s\ell}$ is bad. This would erroneously imply that $\n$ is bad to begin with, and so there exists some prime $p\nmid s$ such that $\n^{sp}$ is good. Replacing the role of $\ell$ with $p$ gives our desired result.
\end{proof}

\begin{thm}\label{thm:good_distinguishing_scrapeClosures}
    There exists $t\mid M$ such that either \begin{enumerate}
        \item $\chi(\ol{\m}^t)\neq\chi(\ol{\n}^t)$ and at least one of $\ol{\m}^t,\ol{\n}^t$ is good.
        \item Exactly one of $\ol{\m}^t,\ol{\n}^t$ is good, and exactly one is bad.
    \end{enumerate}
\end{thm}

\begin{proof}
    By Theorem \ref{thm:distinguishing_scrape_closures}, there is some $s\mid M$ such that $\chi(\ol{\m}^s)\neq\chi(\ol{\n}^s)$, but suppose both $\ol{\m}^s=(a,b,1^{(k-2)})$ and $\ol{\n}^s=(c,d,1^{(k-2)})$ are bad. Let us fix the order of the elements of each so that component-wise $\ol{\m}^s\mid\m=(m_1,m_2,\dots,m_k)$ and $\ol{\n}^s\mid\n=(n_1,n_2,\dots,n_k)$. By the definition of closures, we have $2,3\nmid m_i,n_j$ for each $i,j\geq3$.

    Whenever they exist, consider the scrapes $\m^{ds}$ and $\n^{ds}$, where $d\in\{1,2,3,6\}$, if $d\nmid s$. It must be that $\m^{ds}$ and $\n^{ds}$ remain bad. If $\chi(\m^{ds})\neq\chi(\n^{ds})$, then there are distinct primes $g,h\mid M$ and $i,j,b\geq0$ such that $\m^{g^{i+1}h^b}$ and $\n^{g^{i+1}h^b}$ take on the form of (\ref{eqn:gh_scrapes}), possibly with the roles of $\m$ and $\n$ swapped. Then the result follows from Lemma \ref{lem:good_distinguishing_scrapes}.
    
    Otherwise, we assume that for each $d\in\{1,2,3,6\}$, whenever they exist, that $\chi(\m^{ds})=\chi(\n^{ds})$ and so $\m^{ds}=\n^{ds}$ by Corollary \ref{cor:lessthan2}. However, by the transitivity of scrapes, see Proposition \ref{prop:scrape_properties}, this situation and having $\chi(\ol{m}^s)\neq\chi(\ol{n}^s)$ cannot both occur.
\end{proof}

\subsection{Existence of maximally smooth maps with prescribed factors}

Important for applying Theorem \ref{thm:macbeath}, we now prove the following existence theorem.

\begin{thm}\label{thm:maximal_reps}
    Suppose $|\m|\geq3$ is good. There is an odd prime power $q$ such that if $\x$ is a scrape of $\m$, then there is an $\ol{\x}$-maximally smooth representation of $\De(\m)$ to $\PSL(2,q)$.
\end{thm}

We set up notation for producing such a prime power $q$. Let $X:=\lcm(\x)$ and $M:=\lcm(\m)$, so that $X\mid M$. Define the following subsets of prime divisors of $M$, not counting multiplicity. \begin{align*}
    P &:= \{p \text{ prime}\mid v_p(X)<v_p(M)\},\\
    P_1 &:=\begin{cases}
        P\setminus\{2,3\} & \text{if } 2\nmid X,\ 3\nmid X\\
        P\setminus\{2\} & \text{else if } 2\nmid X\\
        P\setminus\{3\} & \text{else if } 3\nmid X\\
        P & \text{else},
    \end{cases}\\
    P_2 &:=\{\ell \text{ prime}\colon2\ell\mid M\}, \text{ whenever } 2\nmid X \text{ and }v_2(X)<v_2(M),\\
    P_3 &:=\{\ell \text{ prime}\colon3\ell\mid M\}, \text{ whenever } 3\nmid X \text{ and }v_3(X)<v_3(M).
\end{align*}

\begin{lemma}\label{lem:incongruences}
    There exists an odd prime power $q>1$ satisfying the system of congruences and incongruences \begin{align}
        q &\equiv1\pmod{2X},\label{Ncong}\\
        q &\not\equiv\pm1\pmod{2p^{v_p(X)+1}},\text{ for each $p\in P_1$},\label{P1}\\
        q &\not\equiv\pm1\pmod{2\cdot(2p)},\text{ for each $p\in P_2$},\text{ and}\label{P2}\\
        q &\not\equiv\pm1\pmod{2\cdot(3p)},\text{ for each $p\in P_3$}.\label{P3}
    \end{align}
\end{lemma} 

\begin{proof}
    For $p\in P_i,\ i\in\{1,2,3\}$, let $*_p$ denote any of the moduli appearing in (\ref{P1}), (\ref{P2}), and (\ref{P3}). Define $L$ as the least common multiple of $2X$ and each $*_p$. Centering on (\ref{Ncong}), we can write solutions to the entire system of (in)congruences in the form $(1+x2X)+yL$, where $x,y\in\Z$.
    
    We claim that there is some $k\in\Z$ such that $q_y:=(1+k2X)+yL$ is a solution to (\ref{Ncong}), (\ref{P1}), (\ref{P2}), and (\ref{P3}), for each $y\geq0$. In addition, we show $k$ can be chosen so that $\gcd(1+k2X,L)=1$. Dirichlet's Theorem on primes in arithmetic progressions would then provide infinitely many (positive) primes in the form of $q_y$. 

    First, we produce integers $k_p$ for each $p\in P_i$, such that $\gcd(1+k_p2X,p)=1$ and $1+k_p2X\not\equiv\pm1\pmod{*_p}$ via the following procedure.

    \begin{enumerate}
        \item If $p\mid X$, then assume by contradiction that $1+2X\equiv\pm1\pmod{*_p}$. This implies $X\equiv0\pmod{\frac{1}{2}*_p}$ or $X\equiv-1\pmod{\frac{1}{2}*_p}$. The former case $X\equiv 0$ would lead to a false claim of either $2\mid X$, $3\mid X$, or $p^{v_p(X)+1}\nmid X$, depending on the membership of $p\in P_i$; whereas, the latter case $X\equiv-1$ would contradict that $p\mid X$. Hence, it must be that $1+2X\not\equiv\pm1\pmod{*_p}$. Therefore, when $p\mid X$, we may take $k_p=1$, so that $1+2X\not\equiv\pm1\pmod{*_p}$ and $\gcd(1+2X,p)=1$.\label{item:1}
        
        \item If $p\nmid X$, then $1+x2X\equiv1+y2X\pmod{*_p}$ if and only if $ip\mid(x-y)$. Therefore, $$1\not\in\{1+2X,1+4X,\dots,1+(ip-1)2X\}\pmod{*_p},$$ and this set contains all of the congruence classes of the form $1+x2X$, excluding 1. There are exactly $ip-1$ such classes modulo $*_p$, for $i\in\{1,2,3\}$.
        
        Suppose $ip-1\geq4$. Notice that if $p\mid(1+x2X)$, then $p\nmid(1+(x\pm1)2X)$ and $p\nmid(1+(x\pm2)2X)$. Hence, there exists $1\leq k_p\leq ip-1$ such that $1+k_p2X\not\equiv\pm1\pmod{*_p}$ and $\gcd(1+k_p2X,p)=1$.
        
        On the other hand, if $ip-1<4$, then we have $2\in P_1$, $3\in P_1$, or $2\in P_2$. The cases $2,3\in P_1$ would contradict that $p\nmid X$, so we may assume $2\in P_2$. Then, modulo $8=2\cdot4$, consider the 3 element set $\{1+2X,1+4X,1+6X\}$. Notice if $1+x2X\equiv-1\pmod{2\cdot4}$, then $1+(x\pm1)2X\not\equiv-1\pmod{2\cdot4}$. Hence, in this case, there exists $1\leq k_2\leq ip-1=3$ such that $1+k_22X\not\equiv\pm1\pmod{2\cdot4}$ and $\gcd(1+k_22X,2)=1$.

        Lastly, there is the possibility of $p$ belonging the intersection of more than one of the sets $P_1$, $P_2$, or $P_3$. For different $P_i$, our procedure may produce different $1\leq k_p\leq ip-1$ for a fixed $p$; however, we can just take the least such $k_p$.
    \end{enumerate}

    Notice if $p\in P_i$, then for every $j_p\in\Z$, $1+(k_p+j_p\cdot ip)2X\not\equiv\pm1\pmod{*_p}$ and $\gcd(1+(k_p+j_p\cdot ip)2X,p)=1$. Let $k\in\Z$ be a solution to the system of congruences $$k\equiv k_p\pmod{ip},\text{ for each } p\in P_i,$$ which is a consistent system by the Chinese remainder theorem. It then follows that $q:=1+k2X$ satisfies $\gcd(q,L)=1$, as well as (\ref{Ncong}), (\ref{P1}), (\ref{P2}), and (\ref{P3}), which finishes the proof.
\end{proof}

\begin{proof}[Proof of Theorem \ref{thm:maximal_reps}]
    By Theorem \ref{thm:macbeath}, the congruence (\ref{Ncong}) implies there is an $\ol{\x}$-smooth representation $\De(\m)\rightarrow\PSL(2,q)$. Now suppose $\vp\colon \De(\m)\twoheadrightarrow\PSL(2,q)$ is $\cvec$-smooth for some $\cvec\mid\m$. By the definition of $\ol{\x}$-maximality, we need to show that $\chi(\ol{\x})\leq\chi(\cvec)$. Define $C:=\lcm(\cvec),\ X:=\lcm(\x)$, and $M:=\lcm(\m)$.

    Notice if $\cvec\mid\x$, then $\chi(\ol{\x})\leq\chi(\x)\leq\chi(\cvec)$. So we can assume that $\cvec\nmid\x$. By the maximality of the scrape $\x$ as a factor of $\m$, we have $C\nmid X$ by Proposition \ref{prop:scrape_properties}. This implies there is a prime $p$ such that $v_p(X)<v_p(C)$, and in particular, $p\in P$. Define $P_C:=\{p\text{ prime}\mid v_p(X)<v_p(C)\}$, which is a nonempty subset of $P$.

    If there exists $p\in P_C$ such that $p\in P_1$, then (\ref{P1}) would erroneously imply that $\vp$ is not $\cvec$-smooth, since $p^{v_p(X)+1}\not\in\{2,3\}$ would divide some element of $\cvec$, but simultaneously it cannot appear as an order of an element in $\PSL(2,q)$. Hence, we may assume that each $p\in P_C$ satisfies $p\not\in P_1$. Since $P_C\subset P$, this can only occur when $\varnothing\neq P_C\subset\{2,3\}$. Consider the following 3 cases.
    
    \begin{caseof}
        \case{$P_C=\{2\}$}{This implies $2\nmid X$ and $2\mid C$. From (\ref{P2}), it follows that $\PSL(2,q)$ cannot contain elements of order $2p$ for any prime $p\mid C$. Consequently, for $\vp$ to be $\cvec$-smooth, all even elements of $\cvec$ must equal 2. Since $P_C=\{2\}$, we must also have $\frac{C}{2}\mid X$. Since $\x$ is a scrape of $\m$, we have that the 2-scrape $\cvec_{2}$ divides $\x$.

        Putting our findings together, we explicitly write $\cvec=(2^{(r+s)},c_1,\dots,c_t)$, $\cvec_2=(1^{(r+s)},c_1,\dots,c_t)$, which divides component-wise with $\x=(1^{(r)},y_1,\dots,y_s,x_1,\dots,x_t)$, where $c_j\neq2$, $y_i\neq1,2$ and $c_j\mid x_j$. It follows that the closure of $\x$ has the form $\ol{\x}=(2^{(r)},y_1,\dots,y_s,x_1',\dots,x_t')$, so that $\chi(\ol{\x})\leq\chi(\cvec)$.}

        \case{$P_C=\{3\}$}{This implies $3\nmid X$ and $3\mid C$. From (\ref{P3}), it follows that $\PSL(2,q)$ cannot contain elements of order $3p$ for any prime $p\mid C$. Consequently, for $\vp$ to be $\cvec$-smooth, all elements divisible by 3 in $\cvec$ must equal 3. Since $P_C=\{3\}$, we must also have $\frac{C}{3}\mid X$. Since $\x$ is a scrape of $\m$, we have that the 3-scrape $\cvec_{3}$ divides $\x$.

        Putting our findings together, we explicitly write $\cvec=(3^{(r+s)},c_1,\dots,c_t)$, $\cvec_3=(1^{(r+s)},c_1,\dots,c_t)$, which divides component-wise with $\x=(1^{(r)},y_1,\dots,y_s,x_1,\dots,x_t)$, where $c_j\neq3$, $y_i\neq1,3$ and $c_j\mid x_j$. It follows that the closure of $\x$ has the form $\ol{\x}=(3^{(r)},y_1',\dots,y_s',x_1',\dots,x_t')$. This time; however, if any $y_i\in\x$ is equal to 2, by the definition of closure, $y_i'=3\in\ol{\x}$. Hence, $\chi(\ol{\x})\leq\chi(\cvec)$.}

        \case{$P_C=\{2,3\}$}{This implies $2,3\nmid X$ and $6\mid C$. From (\ref{P2}), (\ref{P3}), it follows that $\PSL(2,q)$ can contain neither elements of order $2p$ nor $3p$ for any prime $p\mid C$. Consequently, for $\vp$ to be $\cvec$-smooth, all even elements of $\cvec$ must be equal to 2 and all elements divisible by 3 in $\cvec$ must equal 3. Since $P_C=\{2,3\}$, we must also have $\frac{C}{6}\mid X$. Since $\x$ is a scrape of $\m$, we have that the 6-scrape $\cvec_{6}$ divides $\x$.

        The explicit shapes of $\cvec$, $\cvec_6$, and $\x$ are the concatenations of the corresponding multisets provided in the previous two cases. The resulting shape of $\ol{\x}$ is also a concatenation of the former cases, and leads us to the same conclusion $\chi(\ol{x})\leq\chi(\cvec)$.}
    \end{caseof}

    Hence, $\vp$ is $\ol{\x}$-maximally smooth.
\end{proof}

Finally, the proof of Theorem \ref{thm:distinguishing_factors} follows from Theorems \ref{thm:distinguishing_scrape_closures}, \ref{thm:good_distinguishing_scrapeClosures}, and \ref{thm:maximal_reps}.

\section{Algorithm and effective upperbounds}

In this section, we provide an algorithm for distinguishing two non-isomorphic Fuchsian groups $\Ga$ and $\La$. For producing effective upperbounds on distinguishing finite quotients, we will require bounds on the odd prime powers $q>1$ obtained from Theorem \ref{thm:macbeath}. The following results from Linnik \cite{linnik} and improved by Xylouris \cite{xylouris} is an effective version of Dirichlet's theorem on primes in arithmetic progressions.

\begin{thm}\label{thm:linnik}
    Let $a$ and $D$ be coprime positive integers with $a < D$. Then there exists $k \geq 0$ such that $a+kD$ is prime, and if $p$ is the lowest such prime of this form, then $p<cD^5$, where $c$ is an effectively computable constant which is independent of the choices of $a$ and $D$.
\end{thm}

We now state our main result which establishes both the relative profinite rigidity and the effective upperbound on a distinguishing finite quotient. For positively valued functions $f(x_1,\dots,x_r)$ and $g(x_1,\dots,x_r)$, we denote by $f\ll g$ to mean there exist $K,C>0$ such that $x_i\geq K$ for some $i$, we have $|f(x_1,\dots,x_r)|\leq C|g(x_1,\dots,x_r)|$. We write $f\approx g$ to mean $f\ll g$ and $g\ll f$. 

\begin{thm}\label{thm:effective_bounds}
    Let $\Ga=(g_1;p_1;\m)$ and $\La=(g_2;p_2;\n)$ be non-isomorphic Fuchsian groups, with $k:=|\m|=|\n|\geq0$ and $1,\infty\not\in\m\cup\n$. Set $M:=\lcm(\m)$, $N:=\lcm(\n)$, $L:=\lcm(M,N)$, and $b:=\max\{\bet(\Ga),\bet(\La)\}$. Then there exists a finite group $Q$ having order $$|Q|\ll (L+1)^{15+L^{15(b+k)}}$$ such that $Q$ is a quotient for one of the groups, but not for the other group.
\end{thm}

\subsection{Proof of Theorem \ref{thm:effective_bounds}}

We set $\Ga:=(g_1;p_1;\m)$ and $\La:=(g_2;p_2;\n)$ such that $\Ga\ncong\La$, where $\m:=(m_1,\dots,m_k)$ and $\n:=(n_1,\dots,n_k)$ with neither containing 1 nor $\infty$. We prove Theorem \ref{thm:effective_bounds} by showcasing an algorithm broken up into several cases. We set $M:=\lcm(\m)$, $N:=\lcm(\n)$, $L:=\lcm(M,N)$, and $b:=\max\{\bet(\Ga),\bet(\La)\}$.

\begin{namedthm*}{Different abelianizations}\leavevmode\\
    \indent Suppose $\Ga^\Ab\ncong\La^\Ab$. We consider the following cases.
        
    \begin{enumerate}
        \item \textbf{Case $\bet(\Ga)\neq\bet(\La)$:} Without loss of generality, suppose $\bet(\Ga)>\bet(\La)$, and apply Theorem \ref{thm:homologyTrick} to the surjective maps $\Ga\twoheadrightarrow1$ and $\La\twoheadrightarrow1$. This gives the distinguishing abelian quotient $Q=\Z_a^{\bet(\Ga)}$, where we take $1<a\leq L+1$ to be coprime to $L$. The quotient $Q$ has order $a^{\bet(\Ga)}\leq (L+1)^{\bet(\Ga)}=(L+1)^b$.
            
        \item \textbf{Case $\bet(\Ga)=\bet(\La)$:} In this case, the torsion subgroups $\on{Tor}(\Ga^{\Ab})$ and $\on{Tor}(\La^{\Ab})$ are non-isomorphic. By Proposition \ref{prop:abelianization}, both $|\on{Tor}(\Ga^{\Ab})|$ and $|\on{Tor}(\La^{\Ab})|$ have upperbound $L^k$. Taking $1<a\leq L+1$ coprime to $L$, it follows that one of $Q=\Z_a^{\bet(\Ga)}\times\on{Tor}(\Ga)$ or $Q=\Z_a^{\bet(\La)}\times\on{Tor}(\La)$ must be a distinguishing abelian quotient. This quotient has order $|Q|\leq a^{b}\cdot L^k\leq L^k(L+1)^{b}\leq(L+1)^{b+k}$.
    \end{enumerate}
\end{namedthm*}

\begin{namedthm*}{Same abelianizations}\leavevmode\\
    \indent Suppose $\Ga^\Ab\cong\La^\Ab$, and set $b:=\bet(\Ga)=\bet(\La)$. We proceed with cases for the puncture types of $\Ga$ and $\La$.
        
    \begin{enumerate}
        \item \textbf{$\Ga,\La$ have mixed puncture types:} Let $g\geq0$, $p>0$ and suppose that $\Ga=(g;0;\m)$ and $\La=(0;p;\n)$. We will apply Proposition \ref{prop:different_puncture_types} in each of the following subcases.
            
        \begin{enumerate}
            \item \textbf{$\chi(g;0;\m)\leq\chi(0;p;\n)$:} We start by producing a smooth finite quotient $G$ for $(g;0;\m)$. In the case $\m$ is empty, take $G$ to be the trivial group. Otherwise $\m\neq\varnothing$, and we either use Theorem \ref{thm:smooth_dihedral} to produce a smooth dihedral quotient $G_1$ with order at most $4M$, or we use Theorem \ref{thm:macbeath} to produce a smooth $\PSL(2,q)$-representation with image $G_2$, where $q>1$ is an odd prime power such that $M$ divides one of $\frac{q\pm1}{2}$. By Theorem \ref{thm:linnik}, we have $q\ll M^5$, and so $|G_2|\ll M^{15}$.
            
            In either situation, take an integer $1<a\leq M+1$ that is coprime to $M$. By Theorem \ref{thm:homologyTrick}, we construct the group $Q$, an extension of $G_i$ by the group $\Z_a^f$, where $f:=2-|G_i|\chi(g;0;\m)$. It follows from Proposition \ref{prop:different_puncture_types} that $Q$ is a quotient of $\Ga$ but not for $\La$. Towards an upperbound for the order of $Q$, we have $f\leq2+|G_i|(2g+k-2)\leq |G_i|(b+k)\ll M^{15}(b+k)$. Hence, we have that $|Q|=a^f|G_i|\ll M^{15}(M+1)^{M^{15}(b+k)}\leq(M+1)^{15+M^{15(b+k)}}$.
                    
            \item \textbf{$\chi(g;0;\m)>\chi(0;p;\n)$:} We take $q>1$ to be the smallest odd prime power such that each of $N$ and $L+1$ divides one of $\frac{q-1}{2}$ or $\frac{q+1}{2}$. By Theorem \ref{thm:linnik}, we have $q\ll(L+1)^5N^5\leq(L+1)^{10}$. By Theorem \ref{thm:macbeath}, there is a subgroup $G\leq\PSL(2,q)$ that is a smooth quotient of $(0;p;\n)$, where we map one of the standard parabolic generators of $(0;p;\m)$ to an element in $\PSL(2,q)$ having order $L+1$. This guarantees that $|G|>L$, which is a necessary condition for applying Proposition \ref{prop:different_puncture_types}. The group $G$ has order $|G|\leq|\PSL(2,q)|\ll(L+1)^{30}$. Now, take an integer $1<a\leq L+1$ that is coprime to $L$. By Theorem \ref{thm:homologyTrick}, we construct the group extension $Q$ of the group $G$ by $\Z_a^f$, where $f:=1-|G|\chi(0;p;\n)$. Notice that $f\leq1+|G|(p+k-2)\leq |G|(b+k)\ll (L+1)^{30}(b+k)$ and that $Q$ is a quotient of $(0;p;\n)$, but not for $(g;0;\m)$. Hence, $|Q|=a^f|G|\ll (L+1)^{(L+1)^{30}(b+k)+30}$.
        \end{enumerate}
                
        \item \textbf{$\Ga,\La$ have the same puncture type:} We focus on the punctured and unpunctured cases separately. Since $\Ga^{\Ab}\cong\La^{\Ab}$, the genera and number of punctures must be the same for both Fuchsian groups.
            
        \begin{enumerate}
            \item \textbf{$\Ga=(0;p;\m)$ and $\La=(0;p;\n)$, with $p>0$:} In the case that $0<|\m|=|\n|\leq2$, we can use Corollary \ref{cor:lessthan2} to force $\chi(0;p;\m)\neq\chi(0;p;\n)$; without loss of generality, suppose $\chi(0;p;\m)<\chi(0;p;\n)$. Towards using Theorem \ref{thm:macbeath}, we fix $3-k=3-|\m|$ standard parabolic generators so that they map to elements of order 2; for all other parabolic generators, we will map them trivially. Let $q>1$ be the smallest odd prime power such that each element of $\m$ divides one of $\frac{q-1}{2}$ or $\frac{q+1}{2}$. By Theorem \ref{thm:linnik}, we have $q\ll M^5$. Thus, we have a smooth quotient $G_1\leq\PSL(2,q)$ of $(0;p;\m)$, having order $|G_1|\leq|\PSL(2,q)|\ll M^{15}$. On the other hand, if $|\m|=|\n|\geq3$, we use Theorem \ref{thm:good_distinguishing_scrapeClosures} to obtain $s\mid M$ such that, without loss of generality, $\ol{m_s}$ is good and $\chi(\ol{m_s})<\chi(\ol{n_s})$, or $\ol{m_s}$ is good and $\ol{n_s}$ is bad. We take an odd prime power $q>1$ in accordance to Lemma \ref{lem:incongruences}, which by Theorem \ref{thm:linnik} also has the asymptotic bound $q\ll M^5$. This choice of $q$ produces an $\ol{m_s}$-maximally smooth quotient $G_2\leq\PSL(2,q)$ of $\Ga$, where as a quotient (if applicable) of $\La$, the group $G_2$ can only be at most $\ol{\n_s}$-maximally smooth. We have $|G_2|\leq|\PSL(2,q)|\leq q^3\ll M^{15}$. In either situation, take an integer $1<a\leq M+1$ that is coprime to $M$. By Theorem \ref{thm:homologyTrick}, we construct the group $Q$, an extension of $G_i$ by the group $\Z_a^f$, where $f:=1-|G_i|\chi(0;p;\m)$. It follows that $Q$ is a quotient of $\Ga$ but not for $\La$. Towards an upperbound for the order of $Q$, we have $f\leq1+|G_i|(p+k-2)\leq |G_i|(b+k)\ll M^{15}(b+k)$. Hence, we have that $|Q|=a^f|G_i|\ll M^{15}(M+1)^{M^{15}(b+k)}\leq(M+1)^{15+M^{15(b+k)}}$.
        
            \item \textbf{$\Ga=(g;0;\m)$ and $\La=(g;0;\n)$:} In the case that $0<|\m|=|\n|\leq2$, we can use Corollary \ref{cor:lessthan2} to force $\chi(g;0;\m)\neq\chi(g;0;\n)$; without loss of generality, suppose $\chi(g;0;\m)<\chi(g;0;\n)$. Thus, we can use Theorem \ref{thm:smooth_dihedral} to produce a smooth dihedral quotient $G_1$ of $(g;0;\m)$ with order at most $4M$. On the other hand, if $|\m|=|\n|\geq3$, we use Theorem \ref{thm:good_distinguishing_scrapeClosures} to obtain $s\mid M$ such that, without loss of generality, $\ol{m_s}$ is good and $\chi(\ol{m_s})<\chi(\ol{n_s})$, or $\ol{m_s}$ is good and $\ol{n_s}$ is bad. We take an odd prime power $q>1$ in accordance to Lemma \ref{lem:incongruences}, which by Theorem \ref{thm:linnik} has the asymptotic bound $q\ll M^5$. This choice of $q$ produces an $\ol{m_s}$-maximally smooth quotient $G_2\leq\PSL(2,q)$ of $\Ga$, where as a quotient (if applicable) of $\La$, the group $G_2$ can only be at most $\ol{\n_s}$-maximally smooth. We have $|G_2|\leq|\PSL(2,q)|\leq q^3\ll M^{15}$. In either situation, take an integer $1<a\leq M+1$ that is coprime to $M$. By Theorem \ref{thm:homologyTrick}, we construct the group $Q$, an extension of $G_i$ by the group $\Z_a^f$, where $f:=2-|G_i|\chi(g;0;\m)$. It follows that $Q$ is a quotient of $\Ga$ but not for $\La$. Towards an upperbound for the order of $Q$, we have $f\leq2+|G_i|(2g+k-2)\leq |G_i|(b+k)\ll M^{15}(b+k)$. Hence, we have that $|Q|=a^f|G_i|\ll M^{15}(M+1)^{M^{15}(b+k)}\leq(M+1)^{15+M^{15(b+k)}}$.
        \end{enumerate}
    \end{enumerate}
\end{namedthm*}

After collecting and comparing the bounds produced for each of the cases, we see that the greatest asymptotic upperbound produced is $(L+1)^{15+L^{15(b+k)}}$.

\section{Examples}\label{examples}

For these examples, we used GAP \cite{gap} to facilitate some of our computations.

\begin{ex}
    Let $\Delta=\Delta(4,3,7)$ and $\Gamma=\Delta(2,3,7)$. By Proposition \ref{prop:abelianization}, both $\Delta(4,3,7)$ and $\Delta(2,3,7)$ have trivial abelianizations. Since naturally the group $\Delta(2,3,7)$ is a quotient of $\Delta(4,3,7)$, every quotient of $\Delta(2,3,7)$ must also be a quotient of $\Delta(4,3,7)$. Comparing Euler characteristics yields $-\frac{23}{84}=\chi(4,3,7)<\chi(2,3,7)=-\frac{1}{42}$. From \cite{macbeath}, the group $G=\PSL(2,7)$ has order 168 and is a smooth quotient of $\Delta(4,3,7)$. Taking $a=5$ and $f=2-|G|\chi(4,3,7)=2+168\cdot\frac{23}{84}=48$, then by Theorem \ref{thm:macbeath} there is a group extension $Q$ of $G$ by the abelian group $\Z_5^{48}$ such that $Q$ is a quotient of $\Delta(4,3,7)$, but not for $\Delta(2,3,7)$. The group $Q$ has order $|Q|=5^{48}\cdot168\approx5.97\times10^{35}$, which is roughly 600 decillion.
\end{ex}

\begin{ex}
    Let $\Delta=\Delta(15,42,63)$ and $\Gamma=\Delta(21,21,90)$. These two groups are indistinguishable in many important aspects: both abelianizations are isomorphic to $\Z_3\times\Z_{21}$, both Euler characteristics are equal to $-\frac{563}{630}$, and both groups have non-trivial representations to exactly the same collection of $\PSL(2,q)$'s. It can be shown that the group $G=\PSL(2,11)$ is a quotient for both groups; however, under $\Delta(15,42,63)$, $G$ is $(5,6,3)$-maximally smooth and under $\Delta(21,21,90)$, $G$ is $(3,3,6)$-maximally smooth. Since $-\frac{3}{10}=\chi(5,6,3)<\chi(3,3,6)=-\frac{1}{6}$, we will consider any $(5,6,3)$-maximally smooth map $\pi\col\Delta(15,42,63)\twoheadrightarrow\PSL(2,11)$. Notice that the kernel $K:=\ker\pi$ has signature $(g;0;3^{(132)},7^{(110)},21^{(220)})$, where $g=\frac{1}{2}\bet(K)$. Via the Riemann--Hurwitz formula, the kernel $K:=\ker\pi$ provides a first Betti number of $f=\bet(K)=2-|\PSL(2,11)|\chi(5,6,3)=200$. Taking $a=7$, which is coprime to $|\PSL(2,11)|=660$, we can construct $Q:=\Delta(15,42,63)/K^{(7)}[K,K]$, which is a quotient of $\Delta(15,42,63)$ but not for $\Delta(21,21,90)$. The order of $Q$ is $|Q|=a^f|G|=7^{200}\cdot660\approx6.90\times10^{171}$.
\end{ex}

\begin{ex}
    Let $\Delta=\De(2,3,3,315)$ and $\Gamma=\Delta(15,18,21)$. Both groups have quotients to $G:=\Delta(2,3,3)\cong A_4$; however, it can be shown that the representations $\pi_1\col\Delta\twoheadrightarrow G$ is $(2,3,3,3)$-maximally smooth and that $\pi_2\col\Gamma\twoheadrightarrow G$ is $(3,2,3)$-maximally smooth. Since $\chi(0;0;2,3,3,3)<\chi(0;0;3,2,3)$, we can ensure that $b_1(\ker\pi_1) > b_1(\ker\pi_2)$. The Riemann--Hurwitz formula gives us $b_1(\ker\pi_1)=2-|A_4|\chi(2,3,3,3)=8$. Take $n=2$, and let $Q$ be a extension of $G$ by the abelian group $\Z_2^{8}$. Then $Q$ is a quotient of $\De(2,3,3,315)$ but not a quotient for $\Delta(15,18,21)$. The order of $Q$ is $12\cdot2^{8}=3072$.
\end{ex}


Address: Department of Mathematics and Computer Science, Davidson College, Davidson, NC 28035, USA

Email: \verb"frchan@davidson.edu", \verb"listyron@davidson.edu".



\begin{thebibliography}{9999}

\bibitem{BCR}
M.~R.~Bridson, M.~D.~Conder, A.~W.~Reid, \textit{Determining Fuchsian groups by their finite quotients}, Israel Journal of Mathematics \textbf{214.1} (2016), 1--41.

\bibitem{gap}
The GAP~Group, \textit{GAP -- Groups, Algorithms, and Programming, Version 4.13.1}; 2024, \url{https://www.gap-system.org}.

\bibitem{katok}
S.~Katok, \textit{Fuchsian groups}, University of Chicago press, 1992.

\bibitem{linnik}
U.~V.~Linnik, \textit{On the least prime in an arithmetic progression. I. The basic theorem}, Matematicheskii Sbornik, vol. 15, no. 2, pp. 139--178, 1944.

\bibitem{macbeath}
A.~M.~Macbeath, \textit{Generators of the linear fractional groups}, Proc. Symp. Pure Math \textbf{12} (1969), 14--32.

\bibitem{malcev}
A.~Malcev, \textit{On isomorphic matrix representations of infinite groups}, Matematicheskii Sbornik \textbf{50}, no. 3 (1940), 405--422.

\bibitem{nikolov-segal}
N.~Nikolov and D.~Segal, \textit{On finitely generated profinite groups, I: Strong completeness and uniform bounds}, Annals of Mathematics, pp. 171--238, 2007.

\bibitem{profiniteBook}
L.~Ribes and P.~Zalesskii. \textit{Profinite groups}, Springer, 2000.

\bibitem{selberg}
A.~Selberg, \textit{On discontinuous groups in higher-dimensional symmetric spaces}, in Contributions to Functional Theory, Tata Institute of Fundamental Research, 1960.

\bibitem{thurston}
W.~P.~Thurston, \textit{The geometry and topology of three-manifolds}, Vol. 27. American Mathematical Society, 2022.

\bibitem{mathematica}
Wolfram~Research,~Inc., \textit{Mathematica, Version 14.1}; 2024, \url{https://www.wolfram.com/mathematica}.

\bibitem{xylouris}
T.~Xylouris, \textit{Über die Nullstellen der Dirichletschen L-Funktionen und die kleinste Primzahl in einer arithmetischen Progression}, 2011.

\end{thebibliography}
\end{document}